\documentclass[12pt, reqno]{amsart}
\usepackage{amsfonts,amssymb,latexsym,amsmath, amsxtra,hyperref}
\usepackage{verbatim}
\usepackage{multirow}

\theoremstyle{plain}
\newtheorem{theorem}{Theorem}[section]
\newtheorem*{theorem*}{Theorem}
\newtheorem{corollary}[theorem]{Corollary}

\newtheorem{lemma}[theorem]{Lemma}
\newtheorem{proposition}[theorem]{Proposition}
\newtheorem*{conjecture*}{Conjecture}

\theoremstyle{definition}
\newtheorem*{definition}{Definition}
\theoremstyle{remark}
\newtheorem{remark}[theorem]{Remark}
\newtheorem*{remark*}{Remark}
\newtheorem*{remarks*}{Remarks}

\numberwithin{equation}{section}

\newcommand{\R}{\mathbb R}
\newcommand{\N}{\mathbb N}
\newcommand{\Z}{\mathbb Z}
\newcommand{\C}{\mathbb C}
\newcommand{\HH}{\mathbb H}

\newcommand{\Q}{{\mathbb Q}}

\def\({\left(}
\def\){\right)}

\def\t{\tau}
\def\h{\eta}
\newcommand{\SL}{\operatorname{\operatorname{SL}}}
\def\G{\Gamma}
\def\g{\gamma}
\newcommand{\commenttext}[1]{}

\begin{document}

\title
{Maass-Jacobi Poincar\'e series and Mathieu Moonshine}
\author{Kathrin Bringmann}
\address{Mathematical Institute\\University of
Cologne\\ Weyertal 86-90 \\ 50931 Cologne \\Germany}
\email{kbringma@math.uni-koeln.de}
\author{John Duncan} 
\address{{Department of Mathematics, Applied Mathematics and Statistics}\\Case Western Reserve University\\Cleveland, OH 44106\\U.S.A.}
\email{john.duncan@case.edu}
\author{Larry Rolen}
\address{Mathematical Institute\\University of
Cologne\\ Weyertal 86-90 \\ 50931 Cologne \\Germany}
\email{lrolen@math.uni-koeln.de}
\thanks {The research of the first author was supported by the Alfried Krupp Prize for Young University Teachers of the Krupp foundation and the research leading to these results has received funding from the European Research Council under the European Union's Seventh Framework Programme (FP/2007-2013) / ERC Grant agreement n. 335220 - AQSER.  The research of the second author was supported in part by the Simons Foundation (\#316779) and the U.S. National Science Foundation (DMS 1203162). The third author thanks the University of Cologne and the DFG for their generous support via the University of Cologne postdoc grant DFG Grant D-72133-G-403-151001011, funded under the Institutional Strategy of the University of Cologne within the German Excellence Initiative.} 

\date{\today}


\begin{abstract}
Mathieu moonshine attaches a weak Jacobi form of weight zero and index one to each conjugacy class of the largest sporadic simple group of Mathieu. We introduce a modification of this assignment, whereby weak Jacobi forms are replaced by semi-holomorphic Maass-Jacobi forms of weight one and index two. We prove the convergence of some Maass-Jacobi Poincar\'e series of weight one, and then use these to characterize the semi-holomorphic Maass-Jacobi forms arising from the largest Mathieu group. 
\end{abstract}

\maketitle

\section{Introduction and statement of results}\label{sec:intro}

The Mathieu groups were discovered by Mathieu over 150 years ago \cite{Mat_1861,Mat_1873}. Today, we recognize them as five of the 26 sporadic simple groups, serving as the exceptions to the general rule that most finite simple groups, namely the non-sporadic ones, are either cyclic of prime order, alternating of degree five or more, or finite of Lie type. This is the content of the classification of finite simple groups \cite{MR2072045}. The {largest Mathieu group}, denoted $M_{24}$, may be characterized as the unique (up to isomorphism) proper subgroup of the alternating group of degree $24$ that acts quintuply transitively on $24$ points \cite{MR1010366}. The stabilizer of a point in $M_{24}$ is the sporadic Mathieu group $M_{23}$.

In 1988, Mukai established a surprising role for $M_{23}$ in geometry by showing that it controls, in a certain sense, the finite automorphisms of certain complex manifolds of dimension two. More precisely, he proved \cite{Mukai} that any finite group of symplectic automorphisms of a complex K3 surface is isomorphic to a subgroup of $M_{23}$ that has five orbits in the natural permutation representation on $24$ points. Furthermore, any such subgroup may be realized via symplectic automorphisms of some complex K3 surface. For certain purposes, such as in mathematical physics, K3 surfaces serve as higher dimensional analogues of elliptic curves. For example, they are well-adapted to the constructions of string theory (see e.g. \cite{MR1479699}). All complex K3 surfaces have the same diffeomorphism type, so we may consider one example, such as the Fermat quartic $\{X_1^4+X_2^4+X_3^4+X_4^4=0\}\subset \mathbb{P}^3$, and regard the general K3 surface as a choice of complex structure on its underlying real manifold. Any K3 surface admits a non-vanishing holomorphic two-form, unique up to scale, and an automorphism that acts trivially on this two-form is called symplectic. We refer to \cite{MR2030225,MR785216} for more background on K3 surfaces.

In 2010, Eguchi, Ooguri, and Tachikawa made a stunning observation \cite{Eguchi2010} that relates the largest Mathieu group $M_{24}$ to K3 surfaces. To describe this, let $Z(\tau;z)$ be the unique weak Jacobi form of weight zero and index one satisfying $Z(\tau;0)=24$. It can be shown that $Z$ may be written in the form (throughout $q:=e^{2\pi i \tau}$)
\begin{gather}\label{eqn:intro-ZSCAdec}
Z(\tau;z)=24\mathcal A(\tau;z)\frac{\theta_1(\tau;z)^2}{\eta(\tau)^{3}}+A(q)q^{-\frac18}\frac{\theta_1(\tau;z)^2}{\eta(\tau)^{3}}
\end{gather}
for some series $A(q)=\sum_{n\geq 0}A_nq^n\in\Z[[q]]$,
where $\theta_1$ is the usual Jacobi theta function
\[
	\theta_1(\tau;z):=i\sum_{n\in\Z}(-1)^n\zeta^{\left(n+\frac12\right)}q^{\frac12\left(n+\frac12\right)^2},
\]
$\eta$ denotes the Dedekind eta function, $\eta(\tau):=q^{1/24}\prod_{n\geq1}(1-q^n)$, and $\mathcal A$ is the Lerch sum
\begin{gather}\label{eqn:intro-mu}
	\mathcal A(\tau;z):=\frac{i \zeta^{\frac12}}{\theta_1(\tau;z)}\sum_{n\in\Z}(-1)^n\frac{q^{\frac{n(n+1)}2}\zeta^n}{1-\zeta q^n},
\end{gather}
where $\zeta:=e^{2\pi i z}$ throughout. Note that $\mathcal A(\tau;z)$ is, up to a scalar multiple, the specialization $\mu(z,z;\tau)$ where $\mu$ denotes the Zwegers $\mu$-function \cite{zwegers}. The authors of \cite{Eguchi2010} noticed that each of the first five terms
\begin{gather}\label{eqn:intro-fiveAns}
A_1=90,\ \ \ \  A_2=462,\ \ \ \ A_3=1540,\ \ \ \ A_4=4554,\ \ \ \ A_5=11592,
\end{gather}
of $A$ is twice the dimension of an irreducible representation of $M_{24}$. The terms $A_6$ and $A_7$ admit simple expressions as positive integer combinations of dimensions of irreducible representations of $M_{24}$.  Note that $A_0=-2$, which we may regard as $-2$ times the dimension of the trivial irreducible representation. 

The functions above are motivated on physical grounds. For example, the weak Jacobi form $Z$ is the elliptic genus\footnote{See \cite{MR895567,MR2536791} for background on elliptic genera. An explicit computation of the K3 elliptic genus first appeared in \cite{Eguchi1989}.} of a K3 surface. As a consequence of what Witten has explained in \cite{MR970278}, reformulating elliptic genera in terms of supersymmetric non-linear sigma models, we know that $Z$ can be expressed as a linear combination of characters of unitary irreducible representations of the small $N=4$ super conformal algebra at $c=6$ \cite{Eguchi1987}, and these characters are given by $\mathcal A\theta_1^2\eta^{-3}$ and $q^{n-1/8}\theta_1^2\eta^{-3}$ in \cite{Eguchi1988}; see also \cite{Eguchi2009a}. Thus, in physical terms, the numbers $A_n$ encode (virtual) multiplicities of irreducible representations of the $N=4$ algebra, arising from its action on some Hilbert space.

This connection between (\ref{eqn:intro-fiveAns}) and $M_{24}$ may be compared to the observations of McKay and Thompson \cite{Tho_NmrlgyMonsEllModFn}, which initiated the investigations of Conway and Norton \cite{MR554399}, and the subsequent development\footnote{One may argue that the initiation of monstrous moonshine occurred earlier with Ogg's observation \cite{Ogg_AutCrbMdl} that the primes $p$ dividing the order of the monster are precisely those for which the normalizer of $\Gamma_0(p)$ in $\SL_2(\R)$ defines a genus zero quotient of the upper-half plane.} of monstrous moonshine. If $J(\tau)$ denotes the unique weakly holomorphic modular form of weight zero such that $J(\tau)=q^{-1}+O(q)$, then 
\begin{gather*}\label{eqn:intro-Jqs}
	J(\tau)=q^{-1}+194884q+21493760q^2+\cdots,
\end{gather*}
and we have $196884=1+196883$ and $21493760=1+196883+21296876$, where the right-hand sides are sums of dimensions of irreducible representations of the monster. For an analogy that is somewhat closer, we may replace $J$ with the weight $1/2$ modular form $\eta J$. Upon expanding $\eta(\tau)J(\tau)=q^{1/24}\sum_{n\geq -1}a_nq^n$, we find 
\begin{gather}\label{eqn:intro-fourans}
	a_1=196883,\;\ \ \ \ \ 
	a_2=21296876,\;\ \ \ \ \ 
	a_3=842609326,\; \ \ \ \ \ 
	a_4=19360062527,\;\ \ \ \ \ 
\end{gather}
all of which are degrees of irreducible representations of the monster \cite{ATLAS}. Note that the zeroth term is $a_0=-1$, which we may regard as $-1$ times the dimension of the trivial representation. 

From a purely number theoretic point of view, $\eta J$ may be characterized as the unique weakly holomorphic modular form of weight $1/2$ with multiplier system coinciding with that of $\eta$, and satisfying $\eta(\tau) J(\tau)=q^{-23/24}-q^{1/24}+O(q)$ as $\tau\to i\infty$. It is natural to ask if there is a similar description of $A$. As one might guess from the formula (\ref{eqn:intro-ZSCAdec}), it is better for modular properties\footnote{Motivated by sigma model elliptic genera, Eguchi and Hikami \cite{Eguchi2008} described the modular properties of $H$. Interestingly, the function $H$ appeared earlier in the number theory literature, as one of the examples at the end of \cite{Pri_SmlPosWgt_II}.} to consider $H(\tau)=q^{-1/8}A(q)$. Indeed, this function $H$ is beautifully characterized\footnote{The characterization given here is not hard to prove. We refer the reader to \cite{Cheng2011} for a detailed argument.} as the unique mock modular form of weight $1/2$ with multiplier system coinciding with that of $\eta^{-3}$ satisfying $H(\tau)=q^{-1/8}+O(q)$ as $\tau\to i \infty$. We also have the elegant expression \cite{Eguchi2009a}
\begin{gather}\label{eqn:intro-Hviamu}
	H(\tau)=-8\left(\mathcal A\left(\tau;\tfrac12\right)+\mathcal A\left(\tau;\tfrac{\tau}{2}\right)+\mathcal A\left(\tau;\tfrac{1+\tau}{2}\right)\right)
\end{gather}
 in terms of the Lerch sum $\mathcal A$ of (\ref{eqn:intro-mu}). See (\ref{eqn:defnH}) for an alternative formula for $H$, which is well-adapted to the computation of its coefficients. 
 
 Here we have used the term mock modular form, which refers to a holomorphic function on the upper-half plane whose failure to transform as a modular form is encoded by suitable integrals of an auxiliary function, called the shadow of the mock modular form. The notion arose quite recently from the foundational work of Zwegers \cite{zwegers} and Bruinier and Funke \cite{BruinierFunke}, 
with important further developments supplied by the first author and Ono \cite{BringmannOno2006} and Zagier \cite{zagier_mock}. In this work, we regard mock modular forms as an aspect of the theory of modular invariant (non-holomorphic) harmonic functions on the upper-half plane, by defining them (see \S\ref{sec:jac:mmf}) as the holomorphic parts of harmonic weak Maass forms. The works \cite{Dabholkar:2012nd,OnoCDM,zagier_mock} provide nice introductions to the theory, with \cite{Dabholkar:2012nd} including a discussion of the particular example $H$. 

We have seen evidence that both the monster group and the Mathieu group $M_{24}$ enjoy special, if not mysterious, relationships to (mock) modular forms of weight $1/2$. Alternatively, the weak Jacobi form $Z$ of (\ref{eqn:intro-ZSCAdec}) may serve for $M_{24}$ as an analogue of the elliptic modular invariant $J$. Reformulations such as this have an algebraic significance, for the appearance of dimensions of irreducible representations in (\ref{eqn:intro-fiveAns}) and (\ref{eqn:intro-fourans}) suggests the existence of graded vector spaces with module structures for the corresponding groups. In the case of the monster, this is a pioneering conjecture of Thompson {\cite{Tho_NmrlgyMonsEllModFn}. The construction of such a module is an important step in explaining the original observations, and of course, this construction problem changes depending on which function one considers. For example, the problem varies if one studies $J$ instead of $\eta J$ or $Z$ instead of $H$.

For monstrous moonshine, the existence of a (possibly virtual) monster module with the desired properties was proven by Atkin, Fong, and Smith \cite{MR822245}, although they did not provide an explicit construction. A monster module $V^{\natural}$ with rich algebraic structure was constructed concretely by Frenkel, Lepowsky, and Meurman \cite{FLMPNAS,FLM,FLMBerk}. Soon after this Borcherds introduced the notion of a vertex algebra \cite{Bor_PNAS} and demonstrated that $V^{\natural}$ is an example. Frenkel, Lepowsky, and Meurman then established the conjecture of Thompson in a strong sense, by showing that the monster is precisely the group of vertex algebra automorphisms of $V^{\natural}$ that preserve a certain distinguished vector, called the Virasoro element. This Virasoro element endows $V^{\natural}$ with a module structure for the Virasoro algebra, being the universal central extension of the Lie algebra of polynomial vector fields on the circle $S^1$. The existence of this structure indicates the relevance of $V^{\natural}$ to mathematical physics and to conformal field theory in particular. We refer the reader to \cite{Gaberdiel:1999mc} for more on the role of vertex algebras in conformal field theory.

For the Mathieu group, Gannon \cite{Gannon:2012ck} has proven the existence of an $M_{24}$-module $K=\bigoplus_{n>0}K_{n-1/8}$ such that $H(\tau)=-2q^{-1/8}+\sum_{n>0}\dim K_n q^{n-1/8}$. Moreover, if we define $H_g(\tau)$ for $g\in M_{24}$ by setting
\begin{gather}\label{eqn:intro-defnHg}
H_g(\tau):=-2q^{-\frac18}+\sum_{n>0}({\rm tr}_{K_{n-\frac18}}g)q^{n-\frac18},
\end{gather}
then the functions $H_g$, which are the $M_{24}$ twinings of $H$, are mock modular forms of weight $1/2$ for certain subgroups of $\SL_2(\Z)$ and coincide precisely with the predictions of \cite{MR2793423,Eguchi2010a,Gaberdiel2010a, Gaberdiel2010}. We define the function $H_g$ in a more explicit fashion in \S\ref{sec:M24}. Now consider the functions $Z_g(\tau;z)$, defined by setting
\begin{gather}\label{eqn:intro-ZgHg}
	Z_g(\tau;z):=\chi(g)\mathcal A(\tau;z)\frac{\theta_1(\tau;z)^2}{\eta(\tau)^3}
	+H_g(\tau)\frac{\theta_1(\tau;z)^2}{\eta(\tau)^3},
\end{gather}
where $\chi(g)$ is the number of fixed points of $g$ in the defining permutation representation. Then the functions $Z_g$ are weak Jacobi forms of weight zero and index one for (the same) certain subgroups of $\SL_2(\Z)$ \cite{MR2985326}. Despite this progress, the $M_{24}$-module $K$ has not been constructed explicitly. For this reason, it is important to consider alternative formulations of the Mathieu moonshine observation (\ref{eqn:intro-fiveAns}).

In this article, we consider such an alternative formulation, whereby the weak Jacobi forms $Z_g$ are replaced by semi-holomorphic Maass-Jacobi forms $\phi_g$ of weight one and index two. Before stating our main result, we recall that Maass-Jacobi forms were introduced in \cite{MR2680205} and studied further in \cite{MR2805582}. We give the precise definition in \S\ref{sec:jac}, and invite the reader to regard semi-holomorphic Maass-Jacobi forms as related to the usual Jacobi forms, in much the same way that harmonic weak Maass forms are related to modular forms. 
We refer to \S\ref{sec:jac} also for the notion of principal parts and to \S\ref{sec:M24} for the definition of the $\phi_g$, for $g\in M_{24}$. To state our characterization, we require the theta function $\vartheta^{(1)}_{1,2}(\tau;z)$ (see \S\ref{sec:MaaJacPoi}), defined by setting
\[
\vartheta_{1,2}^{(1)} \left( \tau; z\right) := 
\sum_{\lambda\in\Z} q^{2\lambda^2 } \zeta^{4\lambda} \left(  \zeta q^{\lambda} -  \zeta^{-1}q^{-\lambda}\right).
\]
Our main result is the following characterization theorem.
\begin{theorem}\label{thm:M24:chzn}
Suppose that $\phi$ is a semi-holomorphic Maass-Jacobi form of weight one and index two with the same level and multiplier system as $\phi_g$ for some $g\in M_{24}$. If the principal part of $\phi$ at the infinite cusp is $2\vartheta^{(1)}_{1,2}$ and if the principal parts at non-infinite cusps vanish, then $\phi=\phi_g$.
\end{theorem}

Theorem \ref{thm:M24:chzn} is a strong motivation for a closer consideration of the functions $\phi_g$ and their relationship to $M_{24}$, since it implies that the functions $\phi_g$ may be constructed in a uniform manner as the Maass-Jacobi Poincar\'e series attached to certain easily described multiplier  systems on certain subgroups of the modular group. See Corollary \ref{cor:phigispseries} for a precisely formulated result. From this point of view, our characterization of the $\phi_g$ may be regarded as an analogue of the main result of \cite{Cheng2011}, which gives a uniform construction of the mock modular forms via Rademacher sums. Note that no such construction or characterization holds for the weak Jacobi forms $Z_g$, since they generally have non-vanishing principal parts at cusps other than the infinite one. See \S\ref{sec:M24} for more on this point.
 
Rademacher sums, a kind of regularized Poincar\'e series, play a directly analogous role in other moonshine phenomena as well. For the case of monstrous moonshine, it was proven in \cite{DunFre_RSMG} that the monstrous McKay-Thompson series 
\[
	T_m(\tau):=q^{-1}+\sum_{n>0}\left({\rm tr}_{V^{\natural}_n}m\right)q^n,
\]
obtained by taking the graded trace associated to the action of an element $m$ of the monster group on the monster module $V^{\natural}=\bigoplus_n V^{\natural}_n$, is also a Rademacher sum for every $m$ in the monster, of a very similar kind to that first considered by Rademacher \cite{Rad_FuncEqnModInv} in 1939. 

Umbral moonshine is an extension of the observations recalled here, relating $M_{24}$ to certain mock modular forms. Specifically, it is a family of 23 analogous correspondences indexed by the root systems of the even unimodular positive-definite lattices of rank $24$. In this setting, the case of $M_{24}$ corresponds to $24$ copies of the $A_1$ root system. One of the main conjectures of umbral moonshine \cite{UM,MUM}, still outstanding except for the case of $M_{24}$, is that the mock modular forms arising may all be constructed in a uniform manner via Rademacher sums or regularized Poincar\'e series of a suitable kind. Thus Rademacher sums for the $T_m$ and $H_g$ and Maass-Jacobi Poincar\'e series for the $\phi_g$ allow us to write down uniform constructions of all these functions and conjecturally all the functions of umbral moonshine may also be obtained in this way.

In \cite{DunFre_RSMG}, the identification of the $T_m$ as Rademacher sums has been used to formulate conjectures relating monstrous moonshine to three-dimensional quantum gravity, following earlier work \cite{Dijkgraaf2007,LiSonStr_ChGrav3D,MalWit_QGPtnFn3D,Man_AdS3PFnsRecon,Manschot2007,Moore2007,Witten2007} in the physics literature. Theorem \ref{thm:M24:chzn} and the main result of \cite{Cheng2011} indicate that quantum gravity may, ultimately, also play an important role in explicating the physical origins of Mathieu moonshine, and umbral moonshine more generally, assuming the conjectural construction of the umbral moonshine mock modular forms as Rademacher sums. 

As we have mentioned, there is as yet no known construction of the $M_{24}$-module $K$ in (\ref{eqn:intro-defnHg}). In particular, a physical interpretation of the coefficients of $H=H_e$, in which the role of $M_{24}$ is manifest, is still lacking. The same is true for the semi-holomorphic Maass-Jacobi form $\phi_e$ attached to the identity element of $M_{24}$, but we note the recent work \cite{2013arXiv1307.7717H} which gives a physical interpretation of the closely-related function$\left.\frac{1}{2\pi i}\frac{\partial}{\partial z}\phi_e(\tau;z)\right|_{z=0}$. It would be very interesting, especially given our construction of the $\phi_g$ as Poincar\'e series and the significance of Poincar\'e series in quantum gravity, to see what modification of the methods of \cite{2013arXiv1307.7717H} might furnish a physical interpretation of the functions $\phi_g$.
 
We have discussed the analogy between Theorem \ref{thm:M24:chzn} and the main result of \cite{Cheng2011}. It is worth noting that, on a technical level, our approach enjoys an advantage over that utilized in \cite{Cheng2011}, for the Maass-Jacobi theory employed here allows us to avoid much of the case-by-case analysis that was necessary there. Consequently, our treatment is significantly shorter. For this reason, it is an interesting problem to extend the methods here to the analogues of the mock modular forms $H_g$ that appear in the remaining 22 cases of umbral moonshine. The approach we pursue offers a promising path to establishing other cases of the main conjecture of umbral moonshine: that the mock modular forms arising may be constructed uniformly as Rademacher sums or regularized Poincar\'e series. In the course of our work, we further develop the theory of Maass--Jacobi forms initiated in \cite{MR2680205,MR2805582}, by considering subgroups of the modular group, certain multiplier systems, and larger ranges of weights. In particular, we establish larger regions of convergence for certain Maass--Jacobi Poincar\'e series. 

The paper is organized as follows. In \S\ref{sec:jac} we recall the definitions of Jacobi forms, skew-holomorphic Jacobi forms, Maass-Jacobi forms, and mock modular forms, and explain the relationships between these objects. In \S\ref{sec:MaaJacPoi} we generalize the construction of Poincar\'e series by Richter and the first author \cite{MR2680205, MR2805582} to include higher level, multipliers, and an extended range of weights. Part of the analysis in \S\ref{sec:MaaJacPoi} requires some facts about Gauss sums, which we recall in \S\ref{subsectiongauss}. In \S\ref{sec:decps}, we generalize a useful pairing of Richter and the first author \cite{MR2680205} to prove that the Maass-Jacobi forms we are interested in are determined by their principal parts. Finally, in \S\ref{sec:M24}, we define and characterize the semi-holomorphic Maass-Jacobi forms $\phi_g$ attached to the Mathieu group $M_{24}$ and identify them as Maass-Jacobi Poincar\'e series. We conclude with the proof of Theorem 1.1 and discuss its physical significance and analogy to the weak Jacobi form case. We note that the main difficulty lies in proving convergence of the Fourier expansions of these Poincar\'e series, which requires a subtle analysis.
\section*{Acknowledgements}
\noindent The authors are grateful to Miranda Cheng, Michael Griffin, Jeffrey Harvey, and Ken Ono for useful discussions related to the paper. 

\section{Jacobi forms, skew-holomorphic Jacobi forms, Maass-Jacobi forms, and mock modular forms}\label{sec:jac}

In this section, we recall basic facts about Jacobi forms, Maass-Jacobi forms, and mock modular forms.
\subsection{Jacobi forms and skew-holomorphic Jacobi forms} 

To give the relevant definitions, we require some notation.
Let $\Gamma^J:=\mbox{SL}_2(\Z)\ltimes \Z^2$ be the Jacobi group and for $N\in \mathbb{N}$, let $\Gamma_0^J (N) := \Gamma_0 (N) \ltimes \Z^2\subseteq\Gamma^J$. Recall that $\Gamma^J$ acts naturally on $\HH\times \C$ via
\[
	A(\t;z):=\left(\frac{a\tau+b}{c\tau+d}; \frac{z+\lambda\tau+\mu}{c\tau+d}\right)
\]
for $A=\left[\left(\begin{smallmatrix}a&b\\c&d\end{smallmatrix}\right), \left(\lambda, \mu\right)\right]\in \Gamma^J$, where $\mathbb H:=\left\{\tau\in\C\colon\operatorname{Im}(\tau)>0\right\}$. For $k, m\in\Z$, let
\begin{gather*}
	j_{k,m}(A,(\t;z)):=
	\left(c\tau+d\right)^{-k}
e^{2\pi i m \left(-\frac{c(z+\lambda\tau+\mu)^2}{c\tau+d}+\lambda^2\tau+2\lambda z\right)},\\
	j_{k,m}^{sk}(A,(\t;z)):=
	\left(c\overline{\tau}+d\right)^{1-k}\left|c{\tau}+d\right|^{-1}
e^{2\pi i m \left(-\frac{c(z+\lambda\tau+\mu)^2}{c\tau+d}+\lambda^2\tau+2\lambda z\right)}.
\end{gather*}
We use these to define a weight $k$ and index $m$ action of $\Gamma^J$ on smooth functions 

\noindent$\phi:\HH\times\C\to \C$ by setting for $A\in \Gamma^J$
\begin{gather*}
	\phi|_{k,m}A(\t;z):=j_{k,m}(A,(\t;z))\phi(A(\t;z)).
\end{gather*}
The action $\phi|^{sk}_{k,m}A$ is given similarly but with $j^{sk}_{k,m}$ in place of $j_{k,m}$. 

Next recall that a multiplier system of weight $k$ and index $m$ for $\Gamma_0^J(N)$ is a function $\sigma\colon\Gamma_0^J(N)\to \C$ such that for $A,B\in \G_0^J(N)$
\begin{gather*}
	\sigma(AB)j_{k,m}(AB,(\t;z))=\sigma(A)\sigma(B)j_{k,m}(A,B(\t;z))j_{k,m}(B,(\t;z)).
\end{gather*}
A skew-multiplier system for $\G_0^J(N)$ is defined similarly but with $j^{sk}_{k,m}$ in place of $j_{k,m}$. Given a multiplier system $\sigma$ of weight $k$ and index $m$ for $\G_0^J(N)$ let for $A\in \G_0^J(N)$ 
\begin{gather}\label{eqn:slashA}
\phi|_{\sigma,k,m}A:=\sigma(A)\phi|_{k,m}A,
\end{gather}
and define $\phi|^{sk}_{\sigma,k,m}A$ similarly when $\sigma$ is a skew-multiplier system. 
Note that a more common convention is to define multipliers so that $\sigma^{-1}$ appears in (\ref{eqn:slashA}). Set $\widehat{\Q}:=\Q\cup\{\infty\}$ and call the orbits of $\widehat{\Q}$ under the action $\G_0(N)$ the cusps of $\G_0(N)$. Observe that if $\sigma_0$ is a multiplier system of weight $k$ and index $m$ for $\G^J_0(N)$ and if $\rho\colon\G_0^J(N)\to\C^{\times}$  is a morphism from $\G_0^J(N)$ to the multiplicative group of $\C$, then the product $A\mapsto \sigma(A)=\rho(A)\sigma_0(A)$ is also a multiplier system of weight $k$ and index $m$ for $\G_0^J(N)$.  In this article we are concerned exclusively with the case that $\sigma_0$ is trivial and $\rho\colon\G^J_0(N)\to \C^{\times}$ takes the form $\rho=\rho_{N|h}$ for some $h\in\Z$, where
\begin{gather}\label{eqn:jac:rhoNh}
	\rho_{N|h}\left(\left(\begin{smallmatrix}a&b\\c&d\end{smallmatrix}\right), \left(\lambda, \mu\right)\right)
	:=
	e^{-\frac{2\pi icd}{Nh}}.
\end{gather}
Note that (\ref{eqn:jac:rhoNh}) defines a morphism $\G_0^J(N)\to \C^{\times}$ with kernel $\G_0^J(Nh)$ if $h\vert (24,N)$. 

We need the following generalization of holomorphic Jacobi forms, which allows poles at the cusps.
\begin{definition}
A weakly holomorphic Jacobi form of weight $k$ and index $m$ for $\G^J_0(N)$ with multiplier $\sigma$ is a holomorphic function $\phi\colon\HH\times \C\to \C$ such that the following conditions are satisfied:
\begin{enumerate}
\item We have $\phi|_{\sigma,k,m}A=\phi$ for all $A\in \G_0^J(N)$.
\item For any cusp $\alpha\in \G_0(N)\backslash \widehat{\Q}$ and any $\g\in \text{SL}_2(\Z)$ with $\g\infty\in\alpha$, the function $\phi|_{k,m}[\g,(0,0)]$ admits a Fourier expansion
\begin{gather*}
	\phi|_{k,m}\left[\g,(0,0)\right](\tau;z)=\sum_{\substack{n,r\in\Z\\\frac{4mn}{w}-r^2\gg - \infty}}c_{\phi,\alpha}\left(
		n,r\right)q^{\frac{n}{w}}\zeta^r
\end{gather*}
where $w$ is the width of $\alpha$.
\end{enumerate}
\end{definition} 

\noindent The coefficients $c_{\phi,\alpha}\left(n,r\right)$ depend upon the choice of $\g$, but only up to roots of unity. If $c_{\phi,\alpha}\left({n},r\right)=0$ for $n<0$, for every cusp $\alpha$, then $\phi$ is called a weak Jacobi form. If $c_{\phi,\alpha}\left({n},r\right)=0$ whenever $r^2-4m{n}/{w}>0$ for every cusp $\alpha$, then $\phi$ is called a holomorphic Jacobi form. If $c_{\phi,\alpha}(n,r)=0$ whenever $r^2-4mn/w\geq0$ for every cusp $\alpha$, then  $\phi$ is called a Jacobi cusp form. We denote the spaces of weakly holomorphic Jacobi forms, holomorphic Jacobi forms, Jacobi cusp forms, each of weight $k$ and index $m$ for $\Gamma_0^{J}(N)$ with multiplier $\sigma$, by $J_{\sigma,k,m}^{!}(N)$, $J_{\sigma,k,m}(N)$, and $J_{\sigma,k,m}^{\text{\upshape cusp}}(N)$, respectively.

The notion of weakly skew-holomorphic Jacobi forms  for $\G_0^J(N)$ with multiplier $\sigma$ is formulated similarly except that we require the Fourier expansion at any cusp $\alpha\in \G_0(N)\backslash\widehat{\Q}$ to take the form 
\begin{gather}\label{eqn:jac:skwjacfoucsp}
	\phi|^{sk}_{k,m}\left[\g,(0,0)\right](\tau;z)=\sum_{\substack{n,r\in\Z\\r^2-\frac{4mn}{w}\gg-\infty}}c_{\phi,\alpha}\left(n,r\right)e^{-\frac{\pi v}{m}\left(r^2-\frac{4mn}{w}\right)}q^{\frac{n}{w}}\zeta^r,
\end{gather}
where throughout we write $\tau= u+iv$. If the summation in (\ref{eqn:jac:skwjacfoucsp}) can be restricted to $r^2-4mn/w\geq 0$ for every $\alpha$, then $\phi$ is called a skew-holomorphic Jacobi form for $\G_0^J(N)$ and a cusp form in case $c_{\phi,\alpha}(n,r)=0$ whenever $r^2-4mn/w\leq 0$ for every cusp $\alpha$.

We denote the spaces of weakly skew-holomorphic Jacobi forms, skew-holomorphic Jacobi forms, and skew-holomorphic Jacobi cusp forms, each of weight $k$ and index $m$ for $\Gamma_0^{J}(N)$ with multiplier $\sigma$, by $J_{\sigma,k,m}^{sk!}(N)$, $J_{\sigma,k,m}^{sk}(N)$, and $J_{\sigma,k,m}^{\text{\upshape sk, cusp}}(N)$, respectively.

\subsection{Maass Forms and Mock Modular Forms}\label{sec:jac:mmf}

Maass forms were introduced by Maass and generalized by Bruinier and Funke in \cite{BruinierFunke} to allow growth at the cusps. Following their work, we define a harmonic (weak) Maass form of weight $k\in \frac12 \Z$ for a congruence subgroup $\Gamma$ as follows.
We first recall the usual weight $k$ hyperbolic Laplacian operator given by 
\begin{equation*}
\Delta_k:= -v^2\left( \frac{\partial^2}{\partial u^2}+\frac{\partial^2}{\partial v^2} \right) -ikv\left(  \frac{\partial}{\partial u}  +\frac{\partial}{\partial v} \right).
\end{equation*}
For half-integral $k$, we require that $\Gamma$ has level divisible by four. In this case, we also define $\varepsilon_d$ for odd $d$ by
\begin{equation*}
\varepsilon_d:= \left\{\begin{array}{lr}1 &\textrm{ if }d\equiv 1\pmod{4},\\i&\textrm{ if }d\equiv 3\pmod{4}.\end{array}\right.
\end{equation*}

We may now define harmonic weak Maass forms as follows. 
\begin{definition}\label{MaassDefn} A harmonic weak Maass form of weight $k$ on a congruence subgroup $\Gamma\subseteq\operatorname{SL}_2(\Z)$ (with level $4|N$ if $k\in\frac12+\Z$)  is any $\mathcal C^2$ function $F\colon \mathbb H\to \C$ satisfying:
\begin{enumerate}
\item For all $\gamma\in \Gamma,$ 
\[F(\gamma \tau) = \left\{ \begin{array}{rl} (c\tau+d)^{k}F(\tau)& \mathrm{if } ~k\in \Z,\\ \left(\frac{c}{d}\right)^{2k} \varepsilon_d^{-2k}(c\tau+d)^{k}F(\tau) & \mathrm{if }~ k\in \frac 12+\Z.
 \end{array} \right.\]
\item  We have that \[\Delta_k (F)=0.\]
\item There is a polynomial $P_F(q)=\sum_{n\leq0} C_F(n)q^n\in \C[q^{-1}]$ such that $F(\tau)-P_F(q)=O(e^{-\varepsilon v})$ for some $\varepsilon >0$ as $v\to +\infty$. We require analogous
conditions at all the cusps of $\Gamma$.
\end{enumerate}
\end{definition}
\noindent
We denote the space of weight $k$ harmonic weak Maass forms on $\Gamma$ by $H_k(\Gamma)$. \\
\noindent
{\it Three remarks.}

\noindent
1) 
The space $H_k\left(\Gamma_1(N)\right)$ corresponds to $H_k^+(\Gamma_1(N))$ in the notation of \cite{BruinierFunke}. Bruinier and Funke also defined a larger class of harmonic weak Maass forms by modifying condition (iii) above, however we only need the definition above.\\
\noindent
2) Since holomorphic functions are harmonic, weakly holomorphic modular forms are also harmonic weak Maass forms. 

\noindent
3)  Although the definition above only allows the Maass forms to have level divisible by four, we may analogously define them of arbitrary level by allowing a more general multiplier.

A key property of harmonic weak Maass forms is that they canonically split into a holomorphic piece and a non-holomorphic piece. Namely, if we define
$h(w):=e^{-w}\int_{-2w}^{\infty}e^{-t}t^{-k}dt,$ then we have that any $F\in H_k(\Gamma_1(N))$ (for $k\neq 1$) decomposes as $F=F^++F^-$, where 
 \[F^+(\tau)=\displaystyle\sum_{n\gg-\infty}c_F^+(n)q^n,\quad\quad\quad
F^-(\tau)=\displaystyle\sum_{n<0}c_F^-(n)h(2\pi n v)e(nu),\]
for some complex numbers $c_F^+(n),c_F^-(n)$. 
We refer to $F^+$ as the holomorphic part and $F^-$ as the non-holomorphic part of $F$. We call a $q$-series which is the holomorphic part of some harmonic Maass form a mock modular form. Another crucial operator in the theory of harmonic Maass forms is the $\xi$-operator $\xi_k:=2iv^k\cdot\overline{\frac{\partial}{\partial\overline{\tau}}}$, which has the useful property that it defines a map $\xi_k\colon H_k(\Gamma)\rightarrow S_{2-k}(\Gamma)$. Moreover, Bruinier and Funke \cite{BruinierFunke} proved that the above map is surjective, and following Zagier we define $\xi_k \left(F\right)$ to be the shadow of $F^+$. Finally, we recall that an elementary calculation shows that the non-holomorphic part $F^-$ is actually essentially a period integral of its shadow, specifically:
\[F^-(\tau)=-\left(-2i\right)^{k-1}\int_{-\overline{\tau}}^{i\infty}\frac{\xi_k\left( F(w)\right)}{(\tau+w)^k}dw.\]

\subsection{Maass-Jacobi forms}
We now turn to a special class of Maass-Jacobi forms introduced by Richter and the first author \cite{MR2680205}. 
For the general definition of Maass-Jacobi forms, we refer the reader to \cite{Bringmann21}. Define the Casimir operator by setting
\begin{equation*}
\begin{split}
C^{k,m} := &-2(\tau-\overline{\tau})^2\partial_{\tau\overline{\tau}}-(2k-1)(\tau-\overline{\tau})\partial_{\overline{\tau}}
+ \frac{(\tau-\overline{\tau})^2}{4\pi i m}\partial_{\overline{\tau}zz} \\
& + \frac{k(\tau-\overline{\tau})}{4\pi i m}\partial_{z\overline{z}}+ \frac{(\tau-\overline{\tau})(z-\overline{z})}{4\pi i m}\partial_{zz\overline{z}}-2(\tau-\overline{\tau})(z-\overline{z})\partial_{\tau\overline{z}}+k(z-\overline{z})\partial_{\overline{z}}\\
& +\frac{(\tau-\overline{\tau})^2}{4\pi im}\partial_{\tau\overline{z}\overline{z}}+ \left(\frac{(z-\overline{z})^2}{2}+\frac{k(\tau-\overline{\tau})}{4\pi i m}\right)\partial_{\overline{z}\overline{z}}+ \frac{(\tau-\overline{\tau})(z-\overline{z})}{4\pi im}\partial_{z\overline{z}\overline{z}}\,.\\
\end{split}
\end{equation*} 
Throughout, we note that the notation $\partial_{z\overline{zz}}$, for example, is shorthand for $\partial_z\partial_{\overline z}\partial_{\overline z}$ and $\partial_z$, for example, denotes $\frac{\partial}{\partial z}$.
Note that in the special case that the Jacobi form is holomorphic in $z$, this operator simplifies to
\[
C^{k,m}=-2\left(\tau-\overline{\tau}\right)^2 \partial_{\tau\overline{\tau}}-(2k-1)\left(\tau-\overline{\tau}\right)\partial_{\overline{\tau}}+\frac{\left(\tau-\overline{\tau}\right)^2}{4\pi im}\partial_{\overline{\tau} zz}.
\]
We essentially define a Maass-Jacobi form as a function which transforms as a Jacobi form and is in the kernel of the Casimir operator. More specifically, we make the following definition, where here and throughout we write $z=x+iy$.
\begin{definition}\label{semiholomorphicdefn}
A smooth function $\phi\colon\mathbb{H}\times \C\to\C$ is called a semi-holomorphic Maass-Jacobi form for $\G_0^J(N)$ with multiplier $\sigma$, weight $k$, and index $m$ if the following conditions hold:

\vspace{1ex}
(1) For all $A\in \G_0^J(N), \phi|_{\sigma,k, m} A=\phi$.

\vspace{1ex}

(2) We have $C^{k, m}(\phi)=0$.

\vspace{1ex}

(3) The function $\phi$ is holomorphic in $z$.

\vspace{1ex}

(4) We have that $\phi(\tau; z)=O(e^{ay} e^{\frac{2\pi my^2}{v}})$ as $v\to\infty$ for some $a>0$, and the same\\

\vspace{-.1in}
\noindent
\hspace{.38in} conditions also hold for $\phi|_{\sigma,k,m} A$ with $A\in\Gamma^J$.

\noindent
\end{definition}
The space of semi-holomorphic Maass-Jacobi forms for $\G_0^J(N)$ with 
multiplier $\sigma$, weight $k$, and index $m$ is denoted by $\widehat{\mathbb{J}}_{\sigma,k, m}(N)$.
Skew-holomorphic Jacobi forms and semi-holomorphic Maass-Jacobi forms are related via the following operator
\begin{equation*}
\xi_{k,m}:=\left(\frac{\tau-\overline{\tau}}{2i}\right)^{k-\frac{5}{2}}D_-^{(m)}.
\end{equation*}
Here
$$
D_-^{(m)}:=\left(\frac{\tau-\overline{\tau}}{2i}\right)\left(-(\tau-\overline{\tau})\partial_{\overline{\tau}}-(z-\overline{z})\partial_{\overline{z}}+ \frac{1}{4\pi m}\left(\frac{\tau-\overline{\tau}}{2i}\right)\partial_{\overline{z}\overline{z}}\right)
$$ 
is a ``lowering" operator. That is, if $\phi$ is a smooth function on $\HH\times\C$ and if $A\in\Gamma^J$, then  \cite{MR2680205}:
\begin{equation*}
D_-^{(m)}(\phi)\,\big|_{k-2,m} A =D_-^{(m)}\left(\phi\,\big|_{k,m} A\right) \,.
\end{equation*}
The operator $\xi_{k,m}$ maps semi-holomorphic Maass-Jacobi forms to 
weakly skew-holomorphic Jacobi forms of weight $3-k$ 
\begin{gather*}
	\xi_{k,m}\colon\widehat{\mathbb{J}}_{\sigma,k,m}(N)\to J_{\sigma,3-k,m}^{sk!}(N).
\end{gather*}
\noindent
It is not hard to see that the kernel of $\xi_{k,m}$ is $J_{\sigma, k, m}^! (N)$. 
We define $\widehat{\mathbb{J}}^{\text{\upshape cusp}}_{\sigma,k,m}(N)$ to be the preimage of $J_{\sigma,3-k,m}^{\text{\upshape sk,cusp}}(N)$ under $\xi_{k,m}$.

One can show \cite{MR2680205} that $\phi \in\widehat{\mathbb{J}}_{\sigma, k, m}(N)$ has an expansion of the form
\[
v^{\frac32-k}\sum\limits_{n, r\in\Z\atop{D=0}}c_{\phi}^0(n, r) q^n \zeta^r+\sum\limits_{n, r\in\Z\atop{D\ll\infty}}c_{\phi}^+ (n, r) q^n \zeta^r
+\sum\limits_{n, r\in\Z\atop{D\gg -\infty }}c_{\phi}^-(n, r) h\left(-\frac{\pi Dv}{2m}\right) e\left(\frac{iDv}{4m}\right)q^n \zeta^r,
\]
where $D:=r^2-4mn$. We call the second sum the holomorphic part of $\phi$ and the third sum the non-holomorphic part of $\phi$.
A similar expansion holds for $\phi|_{\sigma, k, m} A$ with $A\in\Gamma^J$.
Moreover, 
\begin{equation*}
\mathbb{P}_{\phi}(\tau;z)=\mathbb{P}_{\phi^+}(\tau;z):=\sum_{\substack{n,r\in\Z \\ D>0 }}
c_{\phi}^+(n,r)q^n\zeta^r
\end{equation*}
is called the principal part of $\phi$ (at infinity). Similarly, $\mathbb{P}_{\phi,\alpha}(\tau;z)$ denotes the principal part at the cusp $\alpha$. We let $\widehat{\mathbb{J}}^{\text{\upshape cusp},\infty}_{\sigma,k,m}(N)\subset\widehat{\mathbb{J}}^{\text{\upshape cusp}}_{\sigma,k,m}(N)$ consists of those $\phi$ for which $\mathbb{P}_{\phi,\alpha}=0$ unless $\alpha$ is the infinite cusp of $\Gamma_0 (N)$.

\section{Gauss sums}\label{subsectiongauss}

In this section, we recall some basic properties of and give some elementary formulas for Gauss sums. These explicit calculations are crucial for proving analytic continuations of the Poincar\'e series analyzed in \S\ref{sec:MaaJacPoi}. Consider the generalized quadratic Gauss sum, defined for $a,b,c\in\Z$ by

\[
G(a, b, c):=\sum_{n=0}^{|c|-1}e_c(an^2+bn),
\]
where $e_c(x):=e^{2\pi ix/c}$. The following lemma describes the most important properties of $G(a,b,c)$ that we need, each of which is well-known.
\begin{lemma}\label{Gausslemma}
For any $a,b,c\in\Z$, the following are true:
\begin{itemize}
\item[(i)]
If $b=0$  and $(a,c)=1$, then we have
\[
G(a, 0, c)=
\begin{cases}
0&\quad\text{ if } c\equiv 2\pmod{4},\\
\varepsilon_c\sqrt{c} \left(\frac{a}{c}\right)&\quad\text{ if } c\text{ is odd},\\
(1+i)\varepsilon_a^{-1}\sqrt{c}\left(\frac{c}{a}\right)&\quad\text{ if } 4|c.
\end{cases}
\]

\item[(ii)]
If $4|c, (a,c)=1,$ and $b$ is odd, then $G(a, b, c)=0$.
\item[(iii)]
For $(c, d)=1$
\[
G(a, b, cd)=G(ac, b, d) G(ad, b, c).
\]
\item[(iv)]
If $(a, c)>1, \text{then }G(a, b, c)=0$ unless $(a, c)|b$, in which case
\[
G(a, b, c)= (a, c) G\left(\frac{a}{(a, c)}, \frac{b}{(a, c)}, \frac{c}{(a, c)}\right).
\]
\end{itemize}
\end{lemma}
For the proof of Theorem \ref{AnaC}, it is very handy to have explicit formulas for the Gauss sums $G\left(2\overline{d}, -r'+\overline{d}, c\right)$, where $\bar d$ is a multiplicative inverse of $d$ modulo $c$. 
For this, we introduce the following auxiliary function, where $\nu:=\operatorname{ord}_2(c)$:
\[\beta_{c,d,r'}:=\begin{cases}1&\text{if }\nu=0, \\ 0&\text{if }\nu\geq1\text{ and }r'\text{ is even,} \\  2&\text{if }\nu=1\text{ and }r'\text{ is odd}, \\  4&\text{if }\nu=2, r'\text{ is odd},\text{ and }\overline{d} \not\equiv r^{\prime} \pmod{4}, \\  0&\text{if }\nu=2, r'\text{ is odd},\text{ and }\overline{d}\equiv r^{\prime} \pmod{4},\\  0&\text{if }\nu\geq3\text, r'\text{ is odd},\text{ and }\overline{d} \not\equiv r^{\prime} \pmod{4},\\ \alpha&\text{if }\nu\geq3, r'\text{ is odd},\text{ and }\overline{d} \equiv r^{\prime} \pmod{4},    \end{cases}\]
where $\alpha:=(1+i) e_{2^{\nu-1}}( -\bar{c}^{\prime}d\frac{1}{16}(\overline{d}-r^{\prime})^2)\varepsilon_{dc^{\prime}}^{-1}\sqrt{2^{\nu-1}}\left(\frac{2^{\nu-1}}{c^{\prime}d}\right)$. The general formula for the Gauss sum is given in the following proposition, whose proof uses Lemma \ref{Gausslemma} and is a long, but elementary calculation.

\begin{proposition}\label{GaussSumComputations}
Let $a,b,c\in\Z$ and $ c = 2^{\nu} c^{\prime}$ with $\nu \in \N_0$ and $c^{\prime}$ odd. Then we have
\begin{equation}\label{foofacts}G\left(2\overline{d},-r^{\prime}+\overline{d},c\right)=e_{c^{\prime}}\left(-\overline{2}^{\nu+3}d\left(\overline{d}-r^{\prime}\right)^2\right)\varepsilon_{c^{\prime}}\sqrt{c^{\prime}}\left(\frac{2^{\nu+1}d}{c^{\prime}}\right)\beta_{c,d,r'}.\end{equation}

\end{proposition}

 \section{Maass-Jacobi Poincar\'e series}\label{sec:MaaJacPoi}
In this section, we generalize the construction of Poincar\'e series by Richter and the first author \cite{MR2680205, MR2805582} to include higher level, multipliers, and an extended range of weights. 
The main difficulty lies in showing convergence.
To define these series, let $M_{\nu,\mu}$ be the usual $M$-Whittaker function, which is a solution to the differential equation
\begin{equation} \label{WhittakerDiff}
\partial_w^2f(w)+
 \left(-\frac{1}{4}+\frac{\nu}{w}+\frac{\frac{1}{4}-\mu^2}{w^2}\right)f(w)=0.
\end{equation}
For $s\in\C$, $\kappa\in\frac12\Z$, and $t \in \R \setminus \{0 \}$, define 
\begin{equation*} 
\mathcal{M}_{s,\kappa} (t):=|t|^{-\frac{\kappa}{2} } \, M_{\text{sgn}(t)\frac{\kappa}{2} ,s-\frac12}(|t|)
\end{equation*}
and 
$$
\phi_{k,m,s}^{(n,r)}(\tau;z):= \mathcal{M}_{s,k-\frac12}
\left(-\frac{\pi Dv}{m}\right)\,
e\left(
rz+\frac{ir^2v}{4m} +nu
\right).
$$
Lemma $1$ of \cite{MR2680205} shows that this function is an eigenfunction of the operator  $ C^{k,m}$  
with eigenvalue
$-2s(1-s)- \frac12\left(k^2-3k +\frac{5}{4}\right) $. For $N$ a positive integer and $h | (N,24)$, we consider the Poincar\'e series
\begin{equation}
\label{Poincare-def}
P_{k,m,N|h,s}^{(n,r)}
(\tau;z) := \frac{2}{\Gamma(2s)} \sum_{A\in\Gamma_{\infty}^J\backslash\Gamma_0^J (N)}\, \left.
\phi_{k,m,s}^{(n,r)} 
\right|_{\sigma,k,m}  A    (\tau;z),
\end{equation}
where $\G$ is the usual gamma function, $\sigma=\rho_{N|h}$ (see (\ref{eqn:jac:rhoNh})), and
$\Gamma_{\infty}^J:=\left\{\left[\left(\begin{smallmatrix}1 & \eta\\
0 &1 \end{smallmatrix}\right),(0,n)\right]\,\,|\,\,\eta,n\in\Z \right\}$. 
Note that we choose a slightly different normalization than in  \cite{MR2680205}.
The estimate
$$
\mathcal{M}_{s,k-\frac12}(t) \ll t^{ \text{Re}(s)-\frac{2k-1}{4} }
\qquad (t \to 0)
$$
yields that $P_{k,m,N|h,s}^{(n,r)}$ is absolutely and uniformly convergent for Re$(s)>\frac54$.
In the cases $s \in \left\{ \frac{k}{2}-\frac14,\frac54-\frac{k}{2} \right\}$, the series $P_{k,m,N|h,s}^{(n,r)}$ are annihilated by $C^{k,m}$. We are particularly interested in the case $k=1$, in which the defining sum is not convergent.

We give the Fourier expansion of $P_{k,m,N|h,s}^{(n,r)}$ in the next theorem after introducing a modified $W$-Whittaker function, and we use this to analytically continue our function. For this, we require further notation.
For $s\in \C$, $\kappa\in\frac12\Z$, and $t\in \R \setminus \{0\}$, set
\begin{equation*} 
\mathcal{W}_{s,\kappa}(t):=
 |t|^{-\frac{\kappa}{2}}
 W_{\text{sgn}(t) \frac{\kappa}{2},\,s-\frac{1}{2}}(|t|).
\end{equation*}
Here $W_{\nu,\mu}$ denotes the usual $W$-Whittaker function, which is also a solution to the differential equation (\ref{WhittakerDiff}).  
Moreover, we define the theta functions 
\begin{gather}
\vartheta_{\kappa,m}^{(r)} \left( \tau; z\right) := \label{eqn:MaaJacPoi:vartheta}
\sum_{\lambda\in\Z} q^{\lambda^2 m} \zeta^{2m\lambda} \left( q^{r\lambda} \zeta^r + (-1)^{\kappa} q^{-r\lambda} \zeta^{-r}\right)
\end{gather}
and the {Jacobi-Kloosterman sums}
\begin{equation*}
K_c(n,r,n',r'; N|h):= e_{2mc}\left(-r r' \right) \sum_{\substack{   d \hspace{-1ex}\pmod{c}^* \\ \lambda  \hspace{-1ex}\pmod{c} } }e^{-\frac{2\pi icd}{Nh}}
e_c \left(
\bar{d} m \lambda^2 +n' d- r' \lambda + \bar{d} n+  \bar{d} r \lambda
\right).
\end{equation*}
Here the $*$ in the summation means that the sum on $d$ only runs over those $d$ modulo $c$ that are coprime to $c$. 

\begin{theorem} \label{sPoincareThm}
We have for $\mathrm{Re}(s)>\frac54$
\begin{equation*} 
P_{k,m,N|h,s}^{(n,r)}(\tau;z)  =  \frac{2}{\Gamma(2s)} q^n
   \mathcal{M}_{s,k-\frac12} \left(-\frac{\pi D v }{m}  \right)
e\left(\frac{i D v}{4m}  \right)
 \vartheta_{k,m}^{(r)} (\tau;z) +
  \sum_{n', r' \in\Z} c_{v,s}(n',r')q^{n'}\, \zeta^{r'},
\end{equation*}
where
$$
c_{v,s}(n',r'):=  b_{v,s}(n',r') +(-1)^k b_{v,s}(n',-r'),
$$
with $b_{v,s}(n',r')$ given as follows ($D':=r'^2-4n'm$):
\begin{enumerate}
\item If $DD'>0$, then $b_{v,s}(n',r')$ equals
\begin{multline*}
  \hspace{5.5ex}
 \frac{2\sqrt{2} \pi i^{-k}  m^{-\frac12}}{\Gamma\left(s-\mathrm{sgn}(D')\frac{2k-1}{4}
 \right)}
  \left(\frac{D'}{D}\right)^{\frac{k}{2}- \frac{3}{4}}
  e^{-\frac{\pi D' v}{2m}} 
  \mathcal{W}_{s,k-\frac12} \left(-\frac{\pi D'v}{m} \right)
  \\
  \times
  \sum_{\substack{c>0\\N|c}}
  c^{-\frac32}K_c(n,r,n',r'; N|h)
  J_{2s-1} \left(\frac{\pi \sqrt{D'D}}{mc}  \right)
,
\end{multline*}
where $J$ is the usual $J$-Bessel function. 
\item If $D'=0$, then $b_{v,s}(n',r')$ equals
 $$
 v^{ \frac{5-2k}{4}-s}
  \frac{2}{ \Gamma(2s) \Gamma\left(s+\frac{2k-1}{4} \right)\Gamma\left(s-\frac{2k-1}{4} \right)}
a_s\left(n', r' \right),
  $$
  where $a_s(n',r')$ is holomorphic for $\operatorname{Re}(s)>\frac54$.
	\item   If $DD'<0$, then $b_{v,s}(n',r')$ equals
  \begin{multline*}
\hspace{5.5ex} 2\sqrt{2} \pi i^{-k}  m^{-\frac12}
 \frac{1}{\Gamma\left(s-\mathrm{sgn}(D')\frac{2k-1}{4}
 \right)}
  \left(\frac{|D'|}{|D|}\right)^{\frac{k}{2}- \frac{3}{4}}
   e^{-\frac{\pi D' v}{2m}}
  \mathcal{W}_{s,k-\frac{1}{2}}\left(\frac{-\pi D'v}{m} \right)
  \\
  \times \sum_{\substack{c>0\\N|c}}
  c^{-\frac32}K_c(n,r,n',r'; N|h)
  I_{2s-1} \left(\frac{\pi \sqrt{|D'D| }}{mc}  \right),
\end{multline*}
where $I$ is the usual $I$-Bessel function.
 \end{enumerate}
\end{theorem}
\begin{proof}
Since the proof of Theorem \ref{sPoincareThm} is very similar to the case $N=1$, which appears as Theorem 4 in \cite{MR2680205}, noting that the multiplier given in \eqref{eqn:jac:rhoNh} only depends on $d\pmod{c}$, we omit it here.
\end{proof}

We next specialize to $s=\frac54 -\frac{k}2$ ($k<0$) for $D>0$.
\begin{corollary}\label{4.2}
For $D>0$ and $k<0$, the functions $P_{k,m,N|h, \frac54-\frac{k}{2}}^{(n,r)}$ are elements of $\widehat{\mathbb{J}}_{\rho_{N| h}, k, m}^{\text{\upshape cusp}, \infty}(N)$ and have Fourier expansions of the form 
\begin{equation*}
\begin{split}
P_{k,m, N|h, \frac54 - \frac{k}2}^{(n,r)}(\tau;z)  = &
 2 q^n  \left( 1-\frac{1}{\Gamma\left(\frac32 -k \right)} \Gamma \left( \frac32 -k, \frac{\pi D v}{m} \right)\right) 
 \vartheta_{k,m}^{(r)} (\tau;z) +
 \sum_{\substack{n', r' \in\Z \\D'<0}} c_{n,r}^{(k)}(n',r') q^{n'}\, \zeta^{r'}  \\ 
 & 
 - 
 \frac{1}{k-\frac32} 
   \left( \frac{\pi D}{m}\right)^{k-\frac32}
 \sum_{\substack{n', r' \in\Z \\ D'>0}} 
  c_{n,r}^{(k)}(n',r')   \Gamma\left(\frac32-k, \frac{\pi D' v}{m} \right)
  q^{n'}\, \zeta^{r'}\\
& + \sum_{\substack{n',r'\in\Z\\D'=0}} c_{n,r}^{(k)} \left( n', r'\right) q^{n'} \zeta^{r'} .
 \end{split}
\end{equation*}

Here $\Gamma(\alpha, x)$ is the usual incomplete Gamma-function, and
$$
c_{n,r}^{(k)}(n',r'):=  b_{n,r}^{(k)}(n',r') +(-1)^k b_{n,r}^{(k)}(n',-r')
$$
with 
$$ b_{n,r}^{(k)}(n',r')=
\left\{\begin{array}{ll} 
  \frac{2 \sqrt{2} \pi i^{-k}  m^{-\frac12}
  \left(\frac{D'}{D}\right)^{\frac{k}{2}- \frac{3}{4}}}{\Gamma \left(\frac32 -k\right)}
\mbox{$\displaystyle  \sum_{c>0\atop{N|c}}$}
  c^{-\frac32}K_c(n,r,n',r'; N|h) 
  J_{\frac32-k} \left(\frac{\pi \sqrt{D'D }}{mc}  \right)
& \text{ if }  D'>0,\\ 
2 \sqrt{2} \pi i^{-k} m^{-\frac12} \left(\frac{|D'|}{D}\right)^{\frac{k}{2}- \frac{3}{4}}
\mbox{$\displaystyle  \sum_{c>0\atop{N|c}}$} 
  c^{-\frac32}K_c(n,r,n',r'; N|h) 
  I_{\frac32-k} \left(\frac{\pi \sqrt{|D'|D }}{mc}  \right)& \text{ if } D'<0,
  \\
  \frac{2}{\Gamma(1-k)\Gamma\left(\frac52-k\right)}\left(a_s\left(n',r'\right)-a_s\left(n',-r'\right)\right)&\text{ if }D'=0.
\end{array}
\right.
$$
In particular, the principal part of $P_{k,m,N|h, \frac54 - k}^{(n,r)}$ equals 
$2 q^n 
 \vartheta_{k,m}^{(r)} (\tau;z)$.
\end{corollary}
\begin{proof}
The proof follows directly from Theorem \ref{sPoincareThm}, using the following special values of the Whittaker functions
 \begin{eqnarray*}
\mathcal{W}_{\frac54-\frac{k}{2},k-\frac12}(y)&=& e^{-\frac{y}{2}},\\
\mathcal{W}_{\frac54-\frac{k}{2},k-\frac12}(-y)&=& e^{\frac{y}{2}} \Gamma \left(\frac32-k,y \right),\\ 
\mathcal{M}_{\frac54-\frac{k}{2},k-\frac12}(-y)&=&\left( k-\frac32\right)e^{\frac{y}{2}} \Gamma \left(\frac32-k,y \right) 
+ e^{\frac{y}{2}} \Gamma \left(\frac52-k \right).
\end{eqnarray*}

The claim that $P_{k,m,N|h, \frac54 - k}^{(n,r)}$ has no principal part at cusps inequivalent to $\infty$ follows by directly taking the limit.
\end{proof}

Our next goal is to show that the Poincar\'e series defined for $s$ with Re$(s)>\frac54$ by (\ref{Poincare-def}) can be  analytically continued by using the Fourier expansions of Corollary \ref{4.2}. The case of most interest to us is $k=1$, for which we require $s=\tfrac54-\tfrac k2=\tfrac34$. 

According to the standard asymptotic formulas for Bessel functions, we may replace $ Y_{\frac12} (\frac{\pi\sqrt{|DD'|}}{2c})$ with
$
\frac{|DD'|^{\frac14}}{c^{\frac12}}
$
for $c$ large, with $Y=I$ or $Y=J$.  
Thus the convergence of the Fourier coefficient $b_{n,r}^{(k)}(n',r')$ of Corollary \ref{4.2} is controlled by that of
\begin{gather}\label{eqn:jacobikloostermansumtobound}
\sum_{c>0}c^{-\frac32}K_c(n,r,n',r'; N|h)\frac{|DD'|^{\frac14}}{c^{\frac12}}.
\end{gather}
Easy estimates show that $K_c(n,r,n',r';N|h)=O(c^{3/2})$ as $c\to \infty$, where the implied constant is independent of $D$ and $D'$. So the sum (\ref{eqn:jacobikloostermansumtobound}) is $O(c^{1-2s})$ and thus converges absolutely for $\operatorname{Re}(s)>1$. In this way Corollary \ref{4.2} leads to a definition of $P_{1,2,N|h,s}^{(0,1)}$ for $\operatorname{Re}(s)>1$. To treat $s=\tfrac34$ requires a more subtle analysis, which we now present.

\begin{theorem}\label{AnaC} The function $P_{1,2,N|h,s}^{(0,1)}$ has an analytic continuation to $s=\frac 34$ given by its Fourier expansion. It is an element in $\widehat{\mathbb{J}}_{\rho_{N|h}, 1, 2}^{\text{cusp}, \infty} (N)$  and its principal part in $\infty$ equals $2 \vartheta^{(1)}_{1,2}$.
\end{theorem}
\begin{remark}
It may be possible to prove the convergence of (\ref{eqn:jacobikloostermansumtobound}) using an adaptation of the spectral theoretic methods of \cite{Roe_EigPrbAutFrmsHypPln,Sel_EstFouCoeffs} (see \cite{Sar_ANT} for a review) to Jacobi forms, together with generalizations of the results of \cite{Goldfeld-Sarnak}. Here, we choose to use elementary methods to rewrite (\ref{eqn:jacobikloostermansumtobound}) in terms of well-known multiplier systems. Our technique is of general interest and should be useful in the case that $m>1$ as well. \end{remark}
\begin{proof}[Proof of Theorem \ref{AnaC}]
The only difficulty  is to show convergence of the Fourier expansion. 
By Theorem \ref{sPoincareThm}, we need to show convergence of 
\begin{equation*}
\sum\limits_{c>0\atop{N|c}}c^{-\frac32} \Big(K_c(0, 1, n', r'; N|h)-K_c(0, 1, n', -r'; N|h)\Big) Y_{\frac12} \left(\frac{\pi\sqrt{|DD'|}}{2c}\right).
\end{equation*}
As above, we may replace $Y_{\frac12} (\frac{\pi\sqrt{|DD'|}}{2c})$ by
$
\frac{|DD'|^{\frac14}}{c^{\frac12}},
$
and aim to bound
\begin{equation}\label{Kloosterbound}
\sum\limits_{c>0\atop{N|c}}\frac1{c^2} \Big(K_c(0, 1, n', r'; N|h)-K_c(0, 1, n', -r'; N|h)\Big).
\end{equation}

We next write the Kloosterman sums in terms of the generalized Gauss sums described in Section \ref{subsectiongauss}, since we can manipulate and compute them directly. 
For this, we write \begin{gather*}
	K_c(0,1,n',r';N|h)=\sqrt{c}\sum_{d\pmod{c}^*}e\left(-\frac{cd}{Nh}\right)e_c(n'd)f(c,d,r')
\end{gather*}
with
 \[f(c,a,r):=\frac{e_{4c}(-r)}{\sqrt c}G(2a,a-r,c).\]

\noindent Now  define \[g_0(c,a,r):=i\left(\frac{-4}r\right)e_{8c}\left(d\left(1-r^2\right)\right)e^{-3\pi is(a,c)},\]
where 
\[s(a,c):=\sum_{n\pmod c}\left(\left(\frac nc\right)\right)\left(\left(\frac{na}c\right)\right)\]
is the usual Dedekind sum and 
\[\left(\left(x\right)\right):=\begin{cases}x-\lfloor x\rfloor-\frac12&\text{ if }x\in\R\backslash\Z,\\ 0&\text{ if }x\in\Z.\end{cases}\]
Assume for now that for $(a,c)=1$ we have the following identity:

\begin{equation}\label{toshow}F(c,a,r):=f(c,a,r)-f(c,a,-r)=-2g_0(c,a,r).\end{equation}
Armed with (\ref{toshow}), the convergence of (\ref{Kloosterbound}) follows from the convergence of the coefficients of the mock modular form $H_g$, established via spectral theory in \cite{Cheng2011}. To see this, assume henceforth that $D'=r'^2-8n'<0$ and define
\begin{gather*}
S(k,c,\varepsilon^{-3}\rho_{N|h}):=\sum_{d\pmod{c}^*}\omega^{-3}_{d,c}e\left(\frac{-cd}{Nh}\right)e\left(\frac{kd}{c}\right),
\end{gather*}
where $\omega_{d,c}:=e^{\pi i s(d,c)}$ and $\varepsilon$ is the multiplier system of $\eta$ \cite{RademacherLogEta}.
Note that 
\[\varepsilon\left(\begin{pmatrix}a&b\\ c&d\end{pmatrix}\right)=\omega_{d,c}e^{\pi i\left(\frac{a+d}{12c}-\frac14\right)},\] 
and $S(k,c,\varepsilon^{-3}\rho_{N|h})$ coincides with the Kloosterman sum denoted $S(0,k,c,\varepsilon^{-3}\rho_{N|h})$ in \cite{Cheng2011}.
Then we have that
\begin{gather*}
\left(\frac{-4}{r'}\right)e\left(-\frac{3}{4}\right)
S\left(\frac{|D'|+1}{8},c,\varepsilon^{-3}\rho_{N|h}\right)
=\sum_{d\pmod{c}^*}e\left(-\frac{cd}{Nh}\right)e\left(\frac{n'd}{c}\right)g_0(c,d,r'),
\end{gather*}
and (\ref{toshow}) implies that
\begin{gather*}
	\frac{1}{\sqrt{c}}\Big(K_c(0,1,n',r';N|h)-K_c(0,1,n',-r';N|h)\Big)
	=
	2\left(\frac{-4}{r'}\right)e\left(\frac{3}{4}\right)
S\left(\frac{|D'|+1}{8},c,\varepsilon^{-3}\rho_{N|h}\right).
\end{gather*}
Thus we have reduced the Jacobi-Maass Kloosterman sums to the (more) classical Kloosterman sums arising from the mock modular forms $H_g$. The convergence of (\ref{Kloosterbound}) then follows directly from Theorem 9.1 of \cite{Cheng2011}.

We now turn to the proof of (\ref{toshow}). Note that we may assume $r$ to be odd, as if $r$ is even, the definition of $g_0(c,a,r)$ and the explicit formulas in Proposition \ref{GaussSumComputations} imply that both sides of (\ref{toshow}) vanish. 
Note that the pairs $(a,c)$ with $(a,c)=1$ are in correspondence with fractions $\frac ac\in\Q$. Thus, it is enough to establish (\ref{toshow}) for $\frac ac=0$ and then to show that both $F$ and $g_0$ transform in the same way under the action of $\operatorname{SL}_2(\Z)$ on $\Q$, as this action is transitive. 
Specifically, it suffices to establish the following:

\begin{enumerate}
\item (Equality at $0$): We have that \begin{equation}\label{stepone}F(1,0,r)=-2g_0(1,0,r).\end{equation} 
\item (Translation property): We have that \begin{equation}\label{steptwo}F(c,a,r)-F(c,a+c,r)= g_0(c,a,r)-g_0(c,a+c,r)=0.\end{equation}
\item (Inversion property): We have that \begin{equation}\label{stepthree}\frac{F(c,a,r)}{F(a,-c,r)}=\frac{g_0(c,a,r)}{g_0(a,-c,r)}.\end{equation}

\end{enumerate}

To show (\ref{stepone}), note that $g_0(1,0,r)=i\left(\frac{-4}r\right)$ as $s(0,1)=0.$ Moreover, \[f(0,1,r)-f(0,1,-r)=e_4(-r)-e_4(r)=-2i\left(\frac{-4}r\right),\] as $r$ is odd. For (\ref{steptwo}), observe that we have $s(a+c,c)=s(a,c)$, as the inverse of $a+c$ modulo $c$ is equal to the inverse of $a$ modulo $c$, and also $G(2a+2c,a+c-r,c)=G(2a,a-r,c)$, which can be seen directly from the definition. 

The last property, (\ref{stepthree}), requires more careful consideration. We may compute the inversion property of $g_0(c,a,r)$ directly, using the classical reciprocity formula for Dedekind sums. Specifically, recall that $s(-a,c)=-s(a,c)$ and for $a,c>0$, both positive, we have the following reciprocity formula, stated in (69.6) of \cite{MR0364103}: 
\begin{equation}\label{DedekindReciprocity}s(a,c)-s(-c,a)=-\frac14+\frac1{12}\left(\frac ca+\frac ac+\frac1{ac}\right).\end{equation}
A short calculation translates the standard Dedekind reciprocity formula (\ref{DedekindReciprocity}) into the transformation equation
\begin{equation}\label{ginversiontrans}g_0(c,a,r)=g_0(a,-c,r)\zeta_8^3e_{8ac}\left(-a^2-c^2-r^2\right),\end{equation}
where $\zeta_a:=e^{2\pi i/a}$. Thus we need to compute the inversion formula for $F(c,a,r)$ for comparison with (\ref{ginversiontrans}). The key is to use an inversion formula for generalized Gauss sums, and then to employ Proposition \ref{GaussSumComputations} on a case-by-case basis.  Since the proof is very similar in each case, we only provide the proof when $2$ exactly divides $c$, 
so we assume this condition on $c$ from now on.
In this case, an elementary calculation, using the exact formula in Proposition \ref{GaussSumComputations}, shows that

\begin{equation}\label{fdiff}F(c,a,r)=(1+\delta_{a,c} i)f(c,a,r)\end{equation}
where \[\delta_{a,c}:=\begin{cases}1&\text{ if } \frac{ac}2\equiv3\pmod4,\\ -1&\text{ if }\frac{ac}4\equiv1\pmod4.\end{cases}\] 

We now give the inversion property of $G(a,b,c)$. From Theorem 1.2.2 of \cite{MR1625181}, we find that for $a,c>0$:
\begin{equation}\label{GaussSumInversion}G(a,b,c)=\sqrt{\frac{c}{2a}}\zeta_8e_{4ac}(-b^2)G\left(-\frac c2,-b,2a\right).\end{equation}
Another elementary calculation gives \begin{equation*}\begin{split}G(2a,a-r,c)=G(2a,-r,c)e_{8c}\left(-\bar2_c^38(a-2r)\right),\end{split}\end{equation*}
where $\bar2_c$ denotes any multiplicative inverse of $2$ modulo $c$.
Using (\ref{fdiff}),(\ref{GaussSumInversion}), and (\ref{foofacts}), we see that
 \begin{equation}\label{bleh1}\begin{split}F(c,a,r)=\frac{e_{4c}(-r)}{2\sqrt a}\zeta_8e_{8ac}\left(-(a-r)^2\right) G\left(-2c,r,a\right)G\left(-\frac {ac}2,r-a,4\right)(1+\delta_{a,c} i)\\= 2\delta_{a,c} \frac{e_{4c}(-r)}{\sqrt a}\zeta_8^3e_{8ac}\left(-(a-r)^2\right)G(-2c,r,a).\end{split}\end{equation}
Using (\ref{fdiff}) to compare the factor $G(-2c,r+c,a)$ to $G(-2c,r,a)$ in $f(a,-c,r)$ gives, after a short calculation,
\begin{equation}\label{bleh2}G(-2c,r+c,a)=\delta_{a,c} G(-2c,r,a)e_{8a}(2r+c).\end{equation}
Thus, (\ref{bleh1}) and (\ref{bleh2}) directly yield
\[f(c,a,r)=f(a,-c,r)\zeta_8^3e_{8ac}\left(-a^2-c^2-r^2\right),\]
as desired.
\end{proof}

\section{Decomposition into Poincar\'e series}\label{sec:decps} 
To show how to decompose the spaces of interest for this paper in terms of Maass-Jacobi Poincar\'e series, we require a certain pairing generalizing the pairings in \cite{MR2805582,BruFun_TwoGmtThtLfts}. Recall from \cite{MR1074485} that
if $\phi, \psi \in J_{\sigma, k,m}^{\text{\upshape sk, cusp}}(N)$, 
then the Petersson scalar product of
$\phi$ and $\psi$ is defined by
\begin{equation*}
\label{skew-product}
\left\langle\phi,\psi \right\rangle:=\frac1{\left[\Gamma_0^J(N):\Gamma^J\right]}
\int_{\Gamma^J (N)\backslash  \HH \times \C
}
\phi(\tau;z)\overline{\psi(\tau;z)}e^{-\frac{4 \pi my^2}{v}} v^kdV,
\end{equation*}
where $dV:= \frac{du dv dx dy}{v^3}$ is the $\Gamma^J$-invariant volume element on $\mathbb{H} \times \C.$
If $\phi\in \widehat{\mathbb{J}}_{\sigma, k,m}^{\text{\upshape cusp}, \infty}(N)$ and $\psi\in J_{\sigma, 3-k,m}^{\text{ \upshape sk, cusp}}(N)$, then we define the pairing
\begin{equation}\label{pairing}
\left\{\phi,\psi \right\}:=\left\langle\xi_{k,m}(\phi),\psi\right\rangle.
\end{equation}
This pairing is determined by the principal part of $\phi$. To be more precise, as in the level one case \cite{MR2805582}, one can show the following. 
\begin{lemma}\label{lem:decps:pairprin}
If $\phi \in \widehat{\mathbb{J}}^{\text{\upshape cusp}, \infty}_{\sigma, k, m} (N)$ and $\psi\in J_{\sigma, 3-k, m}^{\text{\upshape sk, cusp}}(N)$, then we have
\begin{gather*}
\{\phi, \psi\}=\sum_{\substack{D>0 \\ r\pmod{2m}}} c_{\phi}^+(n, r) \overline{c_{\psi}(-n, r)}.
\end{gather*}
\end{lemma}
\begin{proof}
The proof in \cite{MR2805582} follows mutatis mutandis for any multiplier system $\sigma$ of absolute value one. 
\end{proof}
We will now show that Maass-Jacobi forms of weight 1 and index 2 are determined by their principal parts for certain levels $N$, which will be sufficient for our proof of Theorem 1.1.
\begin{proposition}\label{thm:uniq}
If $\phi_1, \phi_2 \in  \widehat{\mathbb{J}}^{\text{\upshape cusp}, \infty}_{\sigma, 1, 2} (N)$ with $\mathbb{P}_{\phi_1} = \mathbb{P}_{\phi_2}$ and $N$ the level of any of the modular forms in Table \ref{table_chiT}, then $\phi_1 = \phi_2$.
\end{proposition}
\begin{proof}
For the proof it is enough to assume that $\mathbb{P}_\phi = 0$ and conclude that then $\phi=0$. From Lemma \ref{lem:decps:pairprin}, we obtain that for $\psi\in  {J}^{\text{\upshape sk, cusp}}_{\sigma,3-1, 2} (N)$ one has 
$$
0= \left\{ \phi, \psi \right\} = \left\langle \xi_{1,2} (\phi), \psi \right\rangle.
$$
We next claim that $\langle \cdot , \cdot \rangle$ is nondegenerate for the levels occurring in Table \ref{table_chiT}.
For the reader's convenience, we recall that the level of a form in this table corresponding to an element $g\in M_{24}$ is equal to $N_g=n_gh_g$, where $n_g$ is the order of $g$ (which is the number in front of the label for $g$ in the first column of Table \ref{table_chiT}) and $h_g$ is as in Table \ref{table_chiT}.

Explicitly, we need to verify that $J_{1,2}(N)=0$ for 
\[
N
\in
\mathcal N
:=
\{
1,2,3,4,5,6,7,8,9,11,14,15,16,20,23,24,36,63,144
\}
.
\]
Thus, for the remainder of the proof, we assume that $N\in\mathcal N$.
  According to Lemma \ref{lem:decps:pairprin}, the function $\psi\mapsto \{\phi,\psi\}$ for $\psi\in J^{sk,cusp}_{2,2}$ is identically zero if $\phi$ has vanishing principal parts at all cusps of $\G_0(N)$. On the other hand, the results of \cite{MR2680205} show that (\ref{pairing}) induces a non-degenerate pairing 
\begin{gather*}
{\mathbb{J}}_{1,2}(N)/J_{1,2}(N)\times J_{2,2}^{sk, cusp}(N)\to\C,
\end{gather*}
so if $\{ \phi,\psi\}=0$ for all $\psi\in J_{2,2}^{sk,cusp}(N)$, then $\phi$ belongs to the space  $J_{1,2}(N)$. The non-degeneracy of (\ref{pairing}) now follows if we can show that $J_{1,2}(N)=\{0\}$.
Suppose that $\phi\in J_{1,2}(N)$. Then we have that
\[
\phi(\tau;z)=h(\tau)q^{\frac18}\vartheta_{1,2}^{(1)}(\tau;z),
\]
where $\vartheta_{1,2}^{(1)}$ was defined in \eqref{eqn:MaaJacPoi:vartheta} and where $h(\tau)$ depends only on $\tau$. 
Hence, we find that $\phi'(\tau;z)$, where the derivative is taken in $z$, is a Jacobi form of index $2$, weight $2$, and level $N$. Hence, the specialization $\phi'(\tau;0)$ lies in $M_2(N)$. A short calculation using the Jacobi triple product then shows that 
\[\phi'(\tau;0)=h(\tau)q^{\frac18}\vartheta_{1,2}^{(1)}\,'(\tau;0)=2h(\tau)\eta^3(\tau)\in M_2(N).\]

Thus, as $\eta^3$ is a cusp form, we find that
\[h(8\tau)\eta^3(8\tau)\in S_2(8N).\]
As $\eta^3(8\tau)$ is an ordinary weight $\frac32$ theta series associated to the Dirichlet character $\left(\frac{-4}{\cdot}\right)$, it is easy to check that $\eta^3(\tau)\in S_{\frac32}(64)$, which implies that 
\[h(8\tau)\in M_{\frac12}\left(\operatorname{lcm}(64,N)\right).\]
Moreover, since $\phi$ has a Fourier expansion with integral powers of $q$, it follows that $h$ has a Fourier expansion with exponents which are all congruent to $\frac78\pmod1$. Hence, $h(8\tau)$ has integral exponents which are all congruent to $7\pmod8$.
However, a quick check, for example using MAGMA, shows that the subspace of $M_{\frac12}\left(\operatorname{lcm}(64,N)\right)$ spanned by forms with Fourier expansions supported on this progression is empty, which implies that $h=0$ and hence that $\phi=0$, as desired.

To complete the proof, we apply the non-degeneracy of $\langle \cdot , \cdot \rangle$ to obtain that $\phi = 0$ since $\mathbb{P}_\phi = 0$.
\end{proof}

\section{Maass-Jacobi forms for $M_{24}$}\label{sec:M24}

In this section we attach a semi-holomorphic Maass-Jacobi form $\phi_g$ to each element $g$ in the sporadic group $M_{24}$. We characterize the Maass-Jacobi forms that arise in this way, and in so doing, identify them as Maass-Jacobi Poincar\'e series. 

In preparation for the definition of $\phi_g$, we recall explicit expressions for the mock modular forms $H_g$ defined rather abstractly in the introduction by (\ref{eqn:intro-defnHg}). For the function $H_e(\tau)=H(\tau)=q^{-1/8}A(q)$ (see (\ref{eqn:intro-ZSCAdec}), (\ref{eqn:intro-Hviamu})), attached to the identity element, we have (see (7.16) of \cite{Dabholkar:2012nd}):
\begin{gather}\label{eqn:defnH}
H(\t) = \frac{-2 E_2(\t) + 48 F_2^{(2)}(\t)}{\h(\t)^3}.
\end{gather}
Here $E_2(\tau):=1-24\sum_{n\geq 1}\sigma_1(n)q^n$ denotes the quasi-modular Eisenstein series of weight $2$, where $ \sigma_1(n):=\sum_{d|n} d$. The function $F^{(2)}_2(\tau)$ is defined by 
\begin{gather*}
F_2^{(2)}(\t):= \sum_{\substack{r>s>0\\ r-s\equiv 1\pmod{ 2} }} (-1)^{r} s q^{\frac{rs}{2}}.
\end{gather*}
As mentioned in the introduction, $H$ is the unique mock modular form of weight $1/2$ for $\text{SL}_2(\Z)$ satisfying $H(\tau)=q^{-1/8}+O(q)$, with multiplier system coinciding with that of $\eta^{-3}$. Recall that $\chi(g)$ denotes the number of fixed points of $g\in M_{24}$, acting in the defining (i.e., unique non-trivial) permutation representation on $24$ points. The functions $H_g$ satisfy \cite{Cheng2011}
\begin{gather}\label{h_g_explicit}
H_g(\t) = \frac{\chi(g)}{24} H(\t) - \frac{ {T}_g(\t)}{\h(\t)^3},
\end{gather}
where the ${T}_g(\tau)$ are certain modular forms of weight two given in Table \ref{table_chiT}.

\begin{table}[h]
\centering  
\caption{Weight two forms attached to $M_{24}$\label{table_chiT}}
\begin{tabular}{rccl}
$[g]$ & $\chi(g)$ & $h_g$& ${T}_g(\tau)$\\
\noalign{\vskip 1mm}
\hline
\noalign{\vskip 1mm}
$1A$ & $24$ &$1$& $0 $\\
$2A$ & ${8}$ &$1$& $16\Lambda_2(\tau)$\\
$2B $&0&2& $2\eta(\tau)^8\eta(2\tau)^{-4}$\\
$3A$&6&1& $6\Lambda_3(\tau)$\\
$3B$&0&3& $2\eta(\tau)^6\eta(3\tau)^{-2}$\\
$4A$&0&2& $2\eta(2\tau)^8\eta(4\tau)^{-4}$\\
$4B$&4&1& $4(-\Lambda_2(\tau)+\Lambda_4(\tau))$\\
$4C$&0&4&$ 2\eta(\tau)^4\eta(2\tau)^2\eta(4\tau)^{-2}$\\
$5A$ &4&1&$ 2\Lambda_5(\tau)$\\
$6A$&2&1&$ 2(-\Lambda_2(\tau)-\Lambda_3(\tau)+\Lambda_6(\tau))$\\
$6B$&0&6&$ 2\eta(\tau)^2\eta(2\tau)^2\eta(3\tau)^2\eta(6\tau)^{-2}$\\
$7AB$&3&1&$ \Lambda_7(\tau)$\\
$8A$&2&1&$ -\Lambda_4(\tau)+\Lambda_8(\tau)$\\
$10A$& 0&2&$ 2\eta(\tau)^3\eta(2\tau)\eta(5\tau)\eta(10\tau)^{-1}$\\
$11A$& 2&1&$ 2(\Lambda_{11}(\tau)-11\eta(\tau)^2\eta(11\tau)^2)/5$\\
$12A$& 0&2&$ 2\eta(\tau)^3\eta(4\tau)^2\eta(6\tau)^3\eta(2\tau)^{-1}\eta(3\tau)^{-1}\eta(12\tau)^{-2}$\\
$12B$& 0&12&$ 2\eta(\tau)^4\eta(4\tau)\eta(6\tau)\eta(2\tau)^{-1}\eta(12\tau)^{-1}$\\
$14AB$& 1&1&$ (-\Lambda_2(\tau)-\Lambda_7(\tau)+\Lambda_{14}(\tau)-{14}\eta(\tau)\eta(2\tau)\eta(7\tau)\eta(14\tau))/3$\\
$15AB$& 1&1&$ (-\Lambda_3(\tau)-\Lambda_5(\tau)+\Lambda_{15}(\tau)-{15}\eta(\tau)\eta(3\tau)\eta(5\tau)\eta(15\tau))/4$\\
$21AB$& 0&3&$ (7\eta(\tau)^3\eta(7\tau)^3\eta(3\tau)^{-1}\eta(21\tau)^{-1}-\eta(\tau)^6\eta(3\tau)^{-2})/3$\\
$23AB$& 1&1&$ (\Lambda_{23}(\tau)-{23}f_{23,a}(\tau)-{69}f_{23,b}(\tau))/11$\\
\end{tabular}
\end{table}

\noindent
The first column of Table \ref{table_chiT} names the conjugacy classes of $M_{24}$ according to the conventions of \cite{ATLAS}. Note that a conjugacy class labelled by $nZ$ consists of elements of order $n$. We condense notation a little by writing $7AB$ as shorthand for $7A\cup 7B$, and similarly for $14AB$, etc. In the last column we have used the notation $\Lambda_M(\t)$ to represent the function 
\begin{gather*}
\Lambda_M(\t) := M q\frac{{\rm d}}{{\rm d}q} \left(\log \frac{\h(M\t)}{\h(\t)}\right)
	=\frac{M(M-1)}{24}+M\sum_{k>0}\sigma_1(k)\left(q^k-Mq^{Mk}\right),
\end{gather*}
which is a modular form of weight two for $\G_0(N)$ if $M|N$. The functions $f_{23,a}$ and $f_{23,b}$ are cusp forms of weight two for $\G_0(23)$, defined by
\begin{gather*}
\begin{split}
	f_{23,a}(\t)&:= \frac{\eta(\tau)^3\eta(23\tau)^3}{\eta(2\tau)\eta(46\tau)}
	+3\h(\t)^2\h(23\t)^2
	+4\eta(\tau)\eta(2\tau)\eta(23\tau)\eta(46\tau)
	+4\eta(2\tau)^2\eta(46\tau)^2, \\
	 f_{23,b}(\t)&:= \h(\t)^2\h(23\t)^2.
\end{split}
\end{gather*}
Note that the definition of $T_{g}$ given here for $g\in 23A\cup 23B$ corrects errors in \cite{MR2985326,Cheng2011}.
One can check, using Table \ref{table_chiT}, that the function $T_g$ is a modular form of weight two for $\Gamma_0(n_g)$, where $n_g:=o(g)$ is the order of $g$, but with a multiplier system $\bar{\rho}_g$ that is generally non-trivial. The third column of Table \ref{table_chiT} specifies this multiplier system, according to the rule that
\begin{gather}\label{eqn:M24-barrho}
	\bar{\rho}_{g}\left(\left(\begin{smallmatrix}a&b\\c&d\end{smallmatrix}\right)\right)
	:=
	e^{-2\pi i\frac{cd}{n_gh_g}}
\end{gather}
for $\left(\begin{smallmatrix}a&b\\c&d\end{smallmatrix}\right)\in\G_0(n_g)$ (see (\ref{eqn:jac:rhoNh})). In terms of the defining permutation representation of $M_{24}$, the value $h_g$ turns out to be the shortest cycle length in an expression for $g$ as a product of disjoint cycles. 
The multiplier system $\bar{\rho}_g$ is trivial on $\Gamma_0(n_gh_g)$ for all $g$ and therefore trivial on $\Gamma_0(n_g)$ exactly when $h_g=1$. Observe that $h_g=1$ if and only if $g$ has a fixed-point in the defining permutation representation. Equivalently, $h_g=1$ if and only if $\chi(g)\neq 0$. Thus, from (\ref{h_g_explicit}), we see that the $g$ for which $\bar{\rho}_g$ is non-trivial are exactly those $g$ for which $H_g$ is a (non-mock) weakly holomorphic modular form of weight $1/2$.

As discussed in \S\ref{sec:jac:mmf}, the function $H_g$ is the holomorphic part of a harmonic weak Maass form $\widehat{H}_g$. For $g\in M_{24}$, the completions $\widehat{H}_g$ are given explicitly {by} 
\begin{gather*}
\widehat{H}_g(\t):=H_g(\t)+\chi(g)(4{i})^{-\frac12} \int_{-\overline{\t}}^{\infty}(w+\t)^{-\frac12}
	\overline{\eta\left(-\overline{w}\right)^3}{\rm d}w.
\end{gather*}
We define the semi-holomorphic Maass-Jacobi forms $\phi_g(\t;z)$ of weight one and index two, for $g\in M_{24}$, by setting
\begin{gather*}
\phi_{g}(\tau;z):=-q^{\frac18}\widehat{H}_g(\t)\vartheta^{(1)}_{1,2}(\t;z),
\end{gather*}
where $\vartheta^{(1)}_{1,2}$ is given in (\ref{eqn:MaaJacPoi:vartheta}). It is not hard to see that $\phi_g \in \widehat{\mathbb{J}}^{\text{\upshape cusp}}_{\sigma_g, 1, 2} (n_g)$ and 
\begin{gather*}
\sigma_g(A):=\rho(A)=\bar{\rho}_{n_g|h_g}(\g)
\end{gather*}
for $A=[\gamma,(\lambda,\mu)]$ (see (\ref{eqn:jac:rhoNh}) and (\ref{eqn:M24-barrho})). In particular, the multiplier of $\phi_g$ is trivial precisely when $\chi(g)\neq 0$.

We now compute the principal parts of the $\phi_g$. In order to do this, we must determine the asymptotic behavior of these weight two forms $T_g$ at each cusp of $\G_0(n_g)$. The necessary data is presented in Table \ref{tab:M24:hatTcusp}. 

\begin{proposition}\label{prop:m24:ppartsphig}
If $g\in M_{24}$, then $\mathbb{P}_{\phi_g}= 2\vartheta_{1,2}^{(1)}$ and $\mathbb{P}_{\phi_g,\alpha}= 0$ for $\alpha$ any non-infinite cusp of $\G_0(n_g)$. In particular, $\phi_g \in \widehat{\mathbb{J}}^{\text{\upshape cusp}, \infty}_{\sigma_g, 1, 2} (n_g)$. 
\end{proposition}
\begin{proof}
From the explicit expression (\ref{h_g_explicit}) for $H_g$, we see that 
\begin{gather*}
\phi^+_g(\tau; z)=\left(-\frac{\chi(g)}{24}H(\tau)+\frac{{T}_g(\t)}{\eta(\t)^3}\right){q^{\frac18}\vartheta^{(1)}_{1,2}} (\tau; z).
\end{gather*}
Inspecting Tables \ref{table_chiT} and \ref{tab:M24:hatTcusp}, we verify that ${T}_g(\tau)=2-\frac{\chi(g)}{12}+O(q)$. Then the identity $\mathbb{P}_{\phi_g}=2\vartheta^{(1)}_{1,2}$ follows, since $H(\tau)q^{\frac{1}{8}}=-2+O(q)$ according to (\ref{eqn:defnH}) and $\frac{T_g}{\eta^3}q^{\frac{1}{8}}$ has the same constant term as $T_g$. 

It remains to verify that $\mathbb{P}_{\phi_g,\alpha}=0$ for $\alpha$ a non-infinite cusp of $\Gamma_0(n_g)$. For this, let $\gamma\in\text{SL}_2(\Z)$ such that $\gamma\infty\in\widehat{\Q}$ is a representative for $\alpha$, and set $A=[\gamma,(0,0)]$. Then we have 
\begin{equation*}
	\phi_g|_{1,2}A
	(\tau; z)
	=\left(-\frac{\chi(g)}{24}\widehat{H}(\tau)+\frac{{T}_g|_2\gamma(\tau)}{\eta^3(\tau)}\right)q^{\frac18}\vartheta^{(1)}_{1,2}(\tau; z).
\end{equation*}
Therefore, $\left(\phi_g|_{1,2}A\right)^+$ is given by $(-\frac{\chi(g)}{24}{H}+\frac{{T}_g|_2\gamma}{\eta^3})q^{\frac{1}{8}}\vartheta^{(1)}_{1,2}$.
Recall the values $h_g$ defined in Table \ref{table_chiT}. Note that ${T}_g|_2\gamma$ is a power series in $q^{\frac1w}$, with $w$ the width of $\alpha$ if $h_g=1$, while if $h_g>1$, then ${T}_g|_2\gamma$ is a power series in $q^{\frac{1}{w'}}$ where $w'$ is the width of the cusp of $\Gamma_0(n_gh_g)$ that is represented by $\gamma\infty$. This is because the multiplier system for ${T}_g$, a modular form for $\Gamma_0(n_g)$, is trivial when restricted to $\Gamma_0(n_gh_g)$. Note that these properties of $T_g$ can be verified directly using Table \ref{table_chiT}. For a given $g$, we now define for $k\in\Z$ values $c(k)$ by requiring
\begin{gather}\label{eqn:M24:FouExpHolCsp}
\left(-\frac{\chi(g)}{24}H(\tau)+\frac{{T}_g|_2\gamma(\tau)}{\eta^3(\tau)}\right)q^{\frac18}=:\sum_{k \gg -\infty} c(k)q^{\frac{k}{w'}}.
\end{gather}
In order to conclude $\mathbb{P}_{\phi_g,\alpha}=0$, we must verify that $c(k)=0$ for $k<w'/8$, since a non-zero term $c(k)q^{k/w'}$ in (\ref{eqn:M24:FouExpHolCsp}) leads to non-zero terms $c(k)q^{2\lambda^2+\lambda+k/w'}\zeta^{4\lambda+1}$ (for example), in the Fourier expansion  of $\left(\phi_g|_{1,2}A\right)^+$. For such terms $D=r^2-4mn/w'=1-8k/w'$ is positive if $k<w'/8$, where $m=2$, $n/w'=2\lambda^2+\lambda+k/w'$, and $r=4\lambda+1$. So we need to check that the expansion of ${T}_g|_2\gamma$ satisfies 
\begin{gather}\label{eqn:M24:Tgatarbcusp}
{T}_g|_2\gamma(\t)=-\frac{\chi(g)}{12}+O\left(q^{\frac{1}{8}}\right),
\end{gather}
whenever $\alpha$ is not the infinite cusp of $\Gamma_0(n_g)$. The identity (\ref{eqn:M24:Tgatarbcusp}) is verified case-by-case by inspecting Tables \ref{table_chiT} and \ref{tab:M24:hatTcusp}. This completes the proof.
\end{proof}
\begin{table}[ht]
\centering  
\caption{Asymptotics of weight two forms attached to $M_{24}$\label{tab:M24:hatTcusp}}
\begin{minipage}[ht]{0.45\linewidth}
\begin{tabular}{rcccl}
$[g]$&$N_g$&$h_g$&cusp rep.&${T}_g$ exp.\\
\noalign{\vskip 1mm}\hline\noalign{\vskip 1mm}
$1A$&$1$&$1$&$\infty$&0\\
\noalign{\vskip 1mm}\hline\noalign{\vskip 1mm}
\multirow{2}{*}{$2A$}&\multirow{2}{*}{$2$}&\multirow{2}{*}{$1$}&$\infty$&$\frac{4}{3}+O(q)$\\
&&&$0$&$-\frac{2}{3}+O(q^{\frac{1}{2}})$\\
\noalign{\vskip 1mm}\hline\noalign{\vskip 1mm}
\multirow{2}{*}{$2B$}&\multirow{2}{*}{$2$}&\multirow{2}{*}{$2$}&$\infty$&$2+O(q)$\\
&&&$0$&$O(q^{\frac{1}{4}})$\\
\noalign{\vskip 1mm}\hline\noalign{\vskip 1mm}
\multirow{2}{*}{$3A$}&\multirow{2}{*}{$3$}&\multirow{2}{*}{$1$}&$\infty$&$\frac{3}{2}+O(q)$\\
&&&$0$&$-\frac{1}{2}+O(q^{\frac{1}{3}})$\\
\noalign{\vskip 1mm}\hline\noalign{\vskip 1mm}
\multirow{2}{*}{$3B$}&\multirow{2}{*}{$3$}&\multirow{2}{*}{$3$}&$\infty$&$2+O(q)$\\
&&&$0$&$O(q^{\frac{2}{9}})$\\
\noalign{\vskip 1mm}\hline\noalign{\vskip 1mm}
\multirow{3}{*}{$4A$}&\multirow{3}{*}{$4$}&\multirow{3}{*}{$2$}&$\infty$&$2+O(q)$\\
&&&$0$&$O(q^{\frac{1}{8}})$\\
&&&$\frac{1}{2}$&$O(q^{\frac{1}{6}})$\\
\noalign{\vskip 1mm}\hline\noalign{\vskip 1mm}
\multirow{3}{*}{$4B$}&\multirow{3}{*}{$4$}&\multirow{3}{*}{$1$}&$\infty$&$\frac{5}{3}+O(q)$\\
&&&$0$&$-\frac{1}{3}+O(q^{\frac{1}{4}})$\\
&&&$\frac{1}{2}$&$-\frac{1}{3}+O(q)$\\
\noalign{\vskip 1mm}\hline\noalign{\vskip 1mm}
\multirow{3}{*}{$4C$}&\multirow{3}{*}{$4$}&\multirow{3}{*}{$4$}&$\infty$&$2+O(q)$\\
&&&$0$&$O(q^{\frac{3}{16}})$\\
&&&$\frac{1}{2}$&$O(q^{\frac{1}{4}})$\\
\noalign{\vskip 1mm}\hline\noalign{\vskip 1mm}
\multirow{2}{*}{$5A$}&\multirow{2}{*}{$5$}&\multirow{2}{*}{$1$}&$\infty$&$\frac{5}{3}+O(q)$\\
&&&$0$&$-\frac{1}{3}+O(q^{\frac{1}{5}})$\\
\noalign{\vskip 1mm}\hline\noalign{\vskip 1mm}
\multirow{4}{*}{$6A$}&\multirow{4}{*}{$6$}&\multirow{4}{*}{$1$}&$\infty$&$\frac{11}{6}+O(q)$\\
&&&$0$&$-\frac{1}{6}+O(q^{\frac{1}{6}})$\\
&&&$\frac{1}{2}$&$-\frac{1}{6}+O(q^{\frac{1}{3}})$\\
&&&$\frac{1}{3}$&$-\frac{1}{6}+O(q^{\frac{1}{2}})$\\
\noalign{\vskip 1mm}\hline\noalign{\vskip 1mm}
\multirow{4}{*}{$6B$}&\multirow{4}{*}{$6$}&\multirow{4}{*}{$6$}&$\infty$&$2+O(q)$\\
&&&$0$&$O(q^{\frac{5}{36}})$\\
&&&$\frac{1}{2}$&$O(q^{\frac{2}{9}})$\\
&&&$\frac{1}{3}$&$O(q^{\frac{1}{4}})$\\
\noalign{\vskip 1mm}\hline\noalign{\vskip 1mm}
\multirow{2}{*}{$7AB$}&\multirow{2}{*}{$7$}&\multirow{2}{*}{$1$}&$\infty$&$\frac{7}{4}+O(q)$\\
&&&$0$&$-\frac{1}{4}+O(q^{\frac{1}{7}})$\\
\noalign{\vskip 1mm}\hline\noalign{\vskip 1mm}
\multirow{4}{*}{$8A$}&\multirow{4}{*}{$8$}&\multirow{4}{*}{$1$}&$\infty$&$\frac{11}{6}+O(q)$\\
&&&$0$&$-\frac{1}{6}+O(q^{\frac{1}{8}})$\\
&&&$\frac{1}{2}$&$-\frac{1}{6}+O(q^{\frac{1}{2}})$\\
&&&$\frac{1}{4}$&$-\frac{1}{6}+O(q)$\\
\noalign{\vskip 1mm}\hline
\end{tabular}
\end{minipage}
\begin{minipage}[ht]{0.45\linewidth}
\begin{tabular}{rcccl}
$[g]$&$N_g$&$h_g$&cusp rep.&${T}_g$ exp.\\
\noalign{\vskip 1mm}\hline\noalign{\vskip 1mm}
\multirow{4}{*}{$10A$}&\multirow{4}{*}{$10$}&\multirow{4}{*}{$2$}&$\infty$&$2+O(q)$\\
&&&$0$&$O(q^{\frac{3}{10}})$\\
&&&$\frac{1}{2}$&$O(q^{\frac{1}{5}})$\\
&&&$\frac{1}{5}$&$O(q^{\frac{1}{4}})$\\
\noalign{\vskip 1mm}\hline\noalign{\vskip 1mm}
\multirow{2}{*}{$11A$}&\multirow{2}{*}{$11$}&\multirow{2}{*}{$1$}&$\infty$&$\frac{11}{6}+O(q)$\\
&&&$0$&$-\frac{1}{6}+O(q^{\frac{2}{11}})$\\
\noalign{\vskip 1mm}\hline\noalign{\vskip 1mm}
\multirow{6}{*}{$12A$}&\multirow{6}{*}{$12$}&\multirow{6}{*}{$2$}&$\infty$&$2+O(q)$\\
&&&$0$&$O(q^{\frac{1}{8}})$\\
&&&$\frac{1}{2}$&$O(q^{\frac{1}{6}})$\\
&&&$\frac{1}{3}$&$O(q^{\frac{1}{8}})$\\
&&&$\frac{1}{4}$&$O(q^{\frac{1}{3}})$\\
&&&$\frac{1}{6}$&$O(q^{\frac{1}{2}})$\\
\noalign{\vskip 1mm}\hline\noalign{\vskip 1mm}
\multirow{6}{*}{$12B$}&\multirow{6}{*}{$12$}&\multirow{6}{*}{$12$}&$\infty$&$2+O(q)$\\
&&&$0$&$O(q^{\frac{23}{144}})$\\
&&&$\frac{1}{2}$&$O(q^{\frac{5}{36}})$\\
&&&$\frac{1}{3}$&$O(q^{\frac{7}{48}})$\\
&&&$\frac{1}{4}$&$O(q^{\frac{2}{9}})$\\
&&&$\frac{1}{6}$&$O(q^{\frac{1}{4}})$\\
\noalign{\vskip 1mm}\hline\noalign{\vskip 1mm}
\multirow{4}{*}{$14AB$}&\multirow{4}{*}{$14$}&\multirow{4}{*}{$1$}&$\infty$&$\frac{23}{12}+O(q)$\\
&&&$0$&$-\frac{1}{12}+O(q^{\frac{1}{7}})$\\
&&&$\frac{1}{2}$&$-\frac{1}{12}+O(q^{\frac{1}{7}})$\\
&&&$\frac{1}{7}$&$-\frac{1}{12}+O(q^{\frac{1}{2}})$\\
\noalign{\vskip 1mm}\hline\noalign{\vskip 1mm}
\multirow{4}{*}{$15AB$}&\multirow{4}{*}{$15$}&\multirow{4}{*}{$1$}&$\infty$&$\frac{23}{12}+O(q)$\\
&&&$0$&$-\frac{1}{12}+O(q^{\frac{2}{15}})$\\
&&&$\frac{1}{3}$&$-\frac{1}{12}+O(q^{\frac{1}{5}})$\\
&&&$\frac{1}{5}$&$-\frac{1}{12}+O(q^{\frac{1}{3}})$\\
\noalign{\vskip 1mm}\hline\noalign{\vskip 1mm}
\multirow{4}{*}{$21AB$}&\multirow{4}{*}{$21$}&\multirow{4}{*}{$3$}&$\infty$&$2+O(q)$\\
&&&$0$&$O(q^{\frac{8}{63}})$\\
&&&$\frac{1}{3}$&$O(q^{\frac{1}{7}})$\\
&&&$\frac{1}{7}$&$O(q^{\frac{2}{9}})$\\
\noalign{\vskip 1mm}\hline\noalign{\vskip 1mm}
\multirow{2}{*}{$23AB$}&\multirow{2}{*}{$23$}&\multirow{2}{*}{$1$}&$\infty$&$\frac{23}{12}+O(q)$\\
&&&$0$&$-\frac{1}{12}+O(q^{\frac{3}{23}})$\\
\noalign{\vskip 1mm}\hline
\end{tabular}
\end{minipage}
\end{table}

We next prove Theorem 1.1, giving a characterization of the semi-holomorphic Maass-Jacobi forms attached to $M_{24}$ by Mathieu moonshine.

\begin{proof}[Proof of Theorem 1.1]

Let $g\in M_{24}$ and suppose that $\phi$ satisfies the hypotheses of the theorem. Then by Proposition \ref{prop:m24:ppartsphig}, $\phi$ and $\phi_g$ have the same principal parts at all cusps of $\G_0(n_g)$ and thus they are equal by Proposition \ref{thm:uniq}.
\end{proof}

Theorem 1.1 admits an important corollary.
Namely, it follows from Theorem \ref{thm:M24:chzn} that the functions $\phi_g$ coincide with Poincar\'e series constructed in Theorem \ref{AnaC}.
\begin{corollary}\label{cor:phigispseries}
If $g\in M_{24}$, then
$\phi_g=P^{(0,1)}_{1,2,n_g|h_g,\frac34}$.
\end{corollary}
\begin{proof}
Let $g\in M_{24}$ and set $\phi:=P^{(0,1)}_{1,2,n_g|h_g,3/4}$. Then $\phi$ belongs to $\widehat{\mathbb{J}}^{\text{\upshape cusp}, \infty}_{\sigma_g, 1, 2} (n_g)$ and satisfies $\mathbb{P}_{\phi}=2\vartheta^{(1)}_{1,2}$, according to Theorem \ref{AnaC}. So $\phi$ satisfies the hypotheses of Theorem \ref{thm:M24:chzn} and thus $\phi=\phi_g$, as required.
\end{proof}
 
The physical significance of Corollary \ref{cor:phigispseries}, stating that the $\phi_g$ may be constructed uniformly via Poincar\'e series, has been discussed in the introduction. We conclude the paper by demonstrating that the analogous statement for the weak Jacobi forms $Z_g$, more closely connected to K3 surfaces via elliptic genera (see (\ref{eqn:intro-ZSCAdec})), does not hold. 
Indeed, in order for $Z_g$ to have vanishing principal part at a non-infinite cusp $\alpha$, we require the condition 
\begin{gather}\label{eqn:M24:TgforZg}
{T}_g|_2\gamma(\t)=-\frac{\chi(g)}{12}+O\left(q^{\frac{1}{4}}\right)
\end{gather} 
to hold for $\gamma\in\SL_2(\Z)$ such that $\gamma\infty$ represents $\alpha$ and this is stronger than (\ref{eqn:M24:Tgatarbcusp}). 
To see that (\ref{eqn:M24:TgforZg}) is necessary, we combine (\ref{eqn:intro-ZgHg}) with (\ref{h_g_explicit}) to obtain 
\begin{gather*}
Z_g(\tau;z)=\frac{\chi(g)}{12}\phi_{0,1}(\tau;z)-T_g(\tau)\phi_{-2,1}(\tau;z),
\end{gather*} 
where $\phi_{0,1}:=\frac12 Z$ and $\phi_{-2,1}:={\theta_1^2}{\eta^{-6}}$ are weak Jacobi forms of index one and with weights zero and negative two, respectively. Recall from \cite{eichler_zagier} that for a weak Jacobi form $\phi$ of index one, we have a Fourier expansion  $$\phi(\tau;z)=\displaystyle\sum_{\substack{n,r\\ n\geq0}}c(n,r)q^n\zeta^r,$$ and further recall that $c(n,r)$ only depends on $D:=r^2-4n$. The polar part of $\phi$ at the infinite cusp is $\sum_{D>0} c(D)q^n\zeta^r$. So from the expansions $\phi_{0,1}(\tau;z)=\zeta^{-1}+10+\zeta+O(q)$ and $\phi_{-2,1}(\tau;z)=\zeta^{-1}-2+\zeta+O(q)$  \cite{eichler_zagier} we conclude that $\phi_{0,1}$ and $\phi_{-2,1}$ have the same prinicipal parts at infinity, namely $\sum_{n\geq 0, D=1}q^n\zeta^r$. 
Now for the expansion of $Z_g$ at a cusp $\alpha$, represented by $\gamma\infty$ for some $\gamma\in \SL_2(\Z)$, we have
\begin{gather*}
Z_g|_{0,1}[\gamma,(0,0)]=\frac{\chi(g)}{12}\phi_{0,1}-\left({T_g}|_{2}\gamma\right)\phi_{-2,1},
\end{gather*} 
since $\phi_{0,1}$ and $\phi_{-2,1}$ are invariant for the relevant actions of $\Gamma^J$. From this expression, we see that a term $c(k)q^{k/w}$ in the Fourier expansion of $T_g|_2\gamma$ contributes a term $c(k)q^{k/w}\zeta$ to the principal part of $Z_g$, whenever $1-4k/w>0$ and $c(k)\neq 0$. In other words, we require $c(k)=0$ whenever $0< \frac{k}{w}< \frac14$ in order for the principal part of $Z_g$ at $\alpha$ to be vanishing. The fact that $\phi_{0,1}$ and $\phi_{-2,1}$ have the same principal parts is what determines the constant term in (\ref{eqn:M24:TgforZg}). Observe now that the condition (\ref{eqn:M24:TgforZg}) fails for $g\in 11A$, for example. For taking $\alpha$ to be the cusp represented by $0$ and taking $\gamma=\left(\begin{smallmatrix}0&-1\\1&0\end{smallmatrix}\right)$ we easily compute, using the explicit expression for $T_g$ in Table \ref{table_chiT}, that the coefficient of $q^{\frac{2}{11}}$ in $T_g|_2\gamma$ is non-zero.

\bigbreak


\clearpage

\begin{thebibliography}{10}

\bibitem{MR2072045}
M. Aschbacher,
\newblock {\em The status of the classification of the finite simple groups,}
\newblock {\em Notices Amer. Math. Soc.} {\bf 51} (2004), no. 7, 736--740.

\bibitem{MR1479699}
P. Aspinwall.
\newblock {\it {$K3$} surfaces and string duality,}
\newblock in {\em Fields, strings and duality ({B}oulder, {CO}, 1996)}, World Sci. Publ., River Edge, NJ (1997), 421--540. 

\bibitem{MR2030225}
W. Barth, K. Hulek, C Peters, and A. Van~de Ven,
\newblock {\em Compact complex surfaces}, volume~4 of {\em Ergebnisse der
  Mathematik und ihrer Grenzgebiete. 3. Folge, A Series of Modern Surveys in
  Mathematics (Results in Mathematics and Related Areas, 3rd Series, A Series
  of Modern Surveys in Mathematics),}
\newblock Springer-Verlag, Berlin, second edition, 2004.

\bibitem{MR785216}
A. Beuville, 
\newblock {\em G\'eom\'etrie des surfaces {$K3$}: modules et p\'eriodes},
\newblock Soci\'et\'e Math\'ematique de France, Paris, 1985.
\newblock Papers from the seminar held in Palaiseau, October 1981--January
  1982, Ast{\'e}risque No. 126 (1985).

\bibitem{MR1625181}
B.. Berndt, R. Evans, and K. Williams,
\newblock {\em Gauss and {J}acobi sums},
\newblock Canadian Mathematical Society Series of Monographs and Advanced
  Texts, John Wiley \& Sons, Inc., New York, 1998.


\bibitem {BruinierFunke}
J. Bruinier and J. Funke,
\newblock {\em On two geometric theta lifts},
\newblock Duke Math J. {\bf 125} (2004), no. 1, 45--90.

\bibitem{Bor_PNAS}
R. Borcherds,
\newblock {\em Vertex algebras, {Kac}-{Moody} algebras, and the {Monster},}
\newblock Proc. Nat. Acad. Sci. U.S.A. 
  {\bf 83} (1986), no. 10, 3068--3071.

\bibitem{BringmannOno2006}
K. Bringmann and K. Ono,
\newblock {\em The {$f(q)$} mock theta function conjecture and partition ranks,}
\newblock Invent. Math. {\bf 165} (2006), no. 2, 243--266.

\bibitem{Bringmann21}
K. Bringmann, M. Raum, and O. Richter.
\newblock {\it Harmonic Maass-Jacobi  forms  with singularities and a theta-like decomposition,}
\newblock Trans. Amer. Math. Soc. {\bf 231} (2012), no. 2, 1100--1118.

\bibitem{MR2680205}
K. Bringmann and O. Richter,
\newblock {\it Zagier-type dualities and lifting maps for harmonic {M}aass-{J}acobi
  forms,}
\newblock Adv. Math. {\bf 225}  (2010), 2298--2315.

\bibitem{MR2805582}
K. Bringmann and O. Richter,
\newblock {\it Exact formulas for coefficients of {J}acobi forms,}
\newblock Int. J. Number Theory {\bf 7} (2011), no. 3, 825--833.

\bibitem{BruFun_TwoGmtThtLfts}
J. Bruinier and J. Funke,
\newblock {\it On two geometric theta lifts,}
\newblock  Duke Math. J. {\bf 125} (2004), no. 1, 45--90.

\bibitem{MR2793423}
M. Cheng,
\newblock {\it {$K3$} surfaces, {$N=4$} dyons and the {M}athieu group {$M_{24}$},}
\newblock Commun. Number Theory Phys. {\bf 4} (2010), no. 4, 623--657.

\bibitem{MR2985326}
M. Cheng and J. Duncan,
\newblock {\it The largest {M}athieu group and (mock) automorphic forms,}
\newblock in {\em String-{M}ath 2011}, Proc. Sympos. Pure
  Math. {\bf 85} (2012), 53--82.

\bibitem{Cheng2011}
M. Cheng and J. Duncan,
\newblock {\it On Rademacher Sums, the Largest Mathieu Group, and the Holographic Modularity of Moonshine,}
\newblock Commun. Number Theory Phys. {\bf 6} (2012), no. 6, 697--758.

\bibitem{UM} M.~Cheng, J. Duncan, and J.~Harvey, \emph{{Umbral Moonshine}}, Commun. Number Theory Phys. \textbf{8} (2014), no.~2.

\bibitem{MUM}
M. Cheng, J. Duncan, and J. Harvey,
\newblock {\it Umbral Moonshine and the Niemeier Lattices}
\newblock Research in the Math. Sci. {\bf 1} (2014), no. 3.

\bibitem{MR554399}
J. Conway and S. Norton,
\newblock {\it Monstrous moonshine,}
\newblock Bull. London Math. Soc. {\bf 11} (1979), no. 3, 308--339.

\bibitem{ATLAS}
J. Conway, R. Curtis, S. Norton, R. Parker, and R. Wilson,
\newblock {\em {Atlas of finite groups. Maximal subgroups and ordinary
  characters for simple groups. With comput. assist. from J. G. Thackray,}}
\newblock Oxford: Clarendon Press, 1985.

\bibitem{MR1010366}
R. Curtis,
\newblock{\em{Natural constructions of the Mathieu groups,}}
\newblock Math. Proc. Cambridge Philos. Soc. {\bf 106} (1989), no. 3, 423--429.

\bibitem{Dabholkar:2012nd}
A. Dabholkar, S. Murthy, and D. Zagier,
\newblock {\it Quantum Black Holes, Wall Crossing, and Mock Modular Forms,}
\newblock Cambridge Monographs in Mathematical Physics, accepted for publication.

\bibitem{Dijkgraaf2007}
R. Dijkgraaf, J. Maldacena, G. Moore, and E. Verlinde,
\newblock {\it A black hole farey tail,}
\newblock preprint.

\bibitem{DunFre_RSMG}
J. Duncan and I. Frenkel,
\newblock {\it Rademacher sums, moonshine and gravity,}
\newblock Commun. Number Theory Phys. {\bf 5} (2011), no. 4, 1--128.

\bibitem{Eguchi2008}
T. Eguchi and K. Hikami,
\newblock {\it Superconformal algebras and mock theta functions,}
\newblock J. Phys. A {\bf 42: 304010} (2009).

\bibitem{Eguchi2009a}
T. Eguchi and K. Hikami,
\newblock {\it Superconformal Algebras and Mock Theta Functions 2. Rademacher Expansion for K3 Surface,}
\newblock Commun. Number Theory Phys. {\bf 3} (2009), 531--554.


\bibitem{Eguchi2010a}
T. Eguchi and K. Hikami,
\newblock {\it Note on Twisted Elliptic Genus of K3 Surface,}
\newblock Phys. Lett. B {\bf 694} (2011), 446--455.

\bibitem{Eguchi2010}
T. Eguchi, H. Ooguri, and Y. Tachikawa,
\newblock {\it Notes on the K3 Surface and the Mathieu group $M_{24}$,}
\newblock Exper. Math. {\bf 20} (2011), 91--96.

\bibitem{Eguchi1989}
T. Eguchi, H. Ooguri, A. Taormina, and S. Yang, 
\newblock {\it Superconformal Algebras and String Compactification on Manifolds with SU(N) Holonomy,}
\newblock Nucl. Phys. B {\bf 315:193} (1989).

\bibitem{Eguchi1987}
T. Eguchi and A. Taormina,
\newblock {\it Unitary representations of the {$N=4$} superconformal algebra,}
\newblock Phys. Lett. B {\bf 196} (1987), no. 1, 75--81.

\bibitem{Eguchi1988}
T. Eguchi and A. Taormina,
\newblock {\it Character formulas for the {$N=4$} superconformal algebra,}
\newblock Phys. Lett. B {\bf  200} (1988), no. 3, 315--322.

\bibitem{eichler_zagier}
M. Eichler and D. Zagier,
\newblock {\em {The theory of Jacobi forms}},
\newblock Birkh{\"a}user Boston, 1985.

\bibitem{FLMPNAS}
I. Frenkel, J. Lepowsky, and A. Meurman,
\newblock {\it A natural representation of the {F}ischer-{G}riess {M}onster with the modular function {$J$} as character,}
\newblock Proc. Nat. Acad. Sci. U.S.A. {\bf 81} (1984), no. 10, 3256--3260.

\bibitem{FLM}
I. Frenkel, J, Lepowsky, and A. Meurman,
\newblock {\em Vertex operator algebras and the {M}onster}, volume 134 of Pure and Applied Mathematics,
\newblock Academic Press Inc., Boston, MA, 1988.

\bibitem{FLMBerk}
I. Frenkel, J. Lepowsky, and A. Meurman,
\newblock {\it A moonshine module for the {M}onster,}
\newblock in {\em Vertex operators in mathematics and physics (Berkeley,
  Calif., 1983)}, Math. Sci. Res. Inst. Publ. {\bf 3},
  231--273. Springer, New York, 1985.

\bibitem{Gaberdiel:1999mc}
M. Gaberdiel,
\newblock {\it An Introduction to conformal field theory,}
\newblock Rept. Prog. Phys. {\bf 63} (2000), 607--667.

\bibitem{Gaberdiel2010a}
M. Gaberdiel, S. Hohenegger, and R. Volpato,
\newblock {\it Mathieu Moonshine in the elliptic genus of K3,}
\newblock J. High Energy Phys. {\bf 1010:062} (2010).

\bibitem{Gaberdiel2010}
M. Gaberdiel, S. Hohenegger, and R. Volpato,
\newblock {\it Mathieu twining characters for K3,}
\newblock  J. High Energy Phys. {\bf 1009:058} (2010).


\bibitem{Gannon:2012ck}
T. Gannon,
\newblock {\it Much ado about Mathieu,} preprint.

\bibitem{Goldfeld-Sarnak}
D. Goldfeld and P. Sarnak,
\newblock{ \it Sums of Kloosterman sums,}
\newblock{Invent. Math. {\bf 71} (1983), no. 2, 243--250.}

\bibitem{2013arXiv1307.7717H}
J. {Harvey} and S. {Murthy},
\newblock {\it Moonshine in Fivebrane Spacetimes}, preprint.

\bibitem{MR970278}
P.~S. Landweber, editor,
\newblock {\em Elliptic curves and modular forms in algebraic topology}, Proceedings of a Conference held at the Institute for Advanced Study Princeton, volume
  1326 of {\em Lecture Notes in Mathematics}, Springer-Verlag, Berlin, 1988.

\bibitem{LiSonStr_ChGrav3D}
W. Li, W. Song, and A. Strominger,
\newblock {\it Chiral gravity in three dimensions.}
\newblock J. High Energy Phys. {\bf 4:082} (2008).

\bibitem{MalWit_QGPtnFn3D}
A. Maloney and E. Witten,
\newblock {\it Quantum gravity partition functions in three dimensions,}
\newblock  J. High Energy Phys. {\bf 2:029} (2010).


\bibitem{Man_AdS3PFnsRecon}
J. Manschot,
\newblock {\it {${\rm AdS}\sb 3$} partition functions reconstructed,}
\newblock J. High Energy Phys. {\bf 10:103} (2007) no. 7.

\bibitem{Manschot2007}
J. Manschot and G. Moore,
\newblock {\it A modern fareytail,}
\newblock Commun. Num. Theor. Phys. {\bf 4} (2010), 103--159.


\bibitem{Mat_1861}
\'E. Mathieu,
\newblock {\it M\'emoire sur l'\'etude des fonctions de plusiers quantit\'es, sur la
  mani\`ere de les former et sur les substitutions qui les laissent
  invariables,}
\newblock J. Math. Pures Appl. {\bf 6} (1861), 241--323.

\bibitem{Mat_1873}
\'E. Mathieu,
\newblock {\it Sur la fonction cinq fois transitive de 24 quantit\'es,}
\newblock J. Math. Pures Appl. {\bf 18} (1873), 25--46.

\bibitem{Moore2007}
G. Moore,
\newblock {\it Les {H}ouches {L}ectures on {S}trings and {A}rithmetic,}
\newblock 2007.

\bibitem{Mukai}
S. Mukai,
\newblock {\it Finite groups of automorphisms of {$K3$} surfaces and the {M}athieu
  group,}
\newblock Invent. Math. {\bf 94} (1988), no. 1, 183--221.

\bibitem{MR895567}
S. Ochanine,
\newblock {\it Sur les genres multiplicatifs d\'efinis par des int\'egrales
  elliptiques,}
\newblock Topology {\bf 26} (1987), no. 2, 143--151.

\bibitem{MR2536791}
S. Ochanine,
\newblock {\it What is\dots an elliptic genus,}
\newblock  Notices Amer. Math. Soc. {\bf 56} (2009), no. 6, 720--721.

\bibitem{Ogg_AutCrbMdl}
A. Ogg,
\newblock {\it Automorphismes de courbes modulaires,}
\newblock In {\em S\'eminaire {D}elange-{P}isot{P}oitou (16e ann\'ee: 1974/75),
  {T}h\'eorie des nombres, {F}asc. 1, {E}xp. {N}o. 7}, page~8. Secr\'etariat
  Math\'ematique, Paris, 1975.

\bibitem{OnoCDM}
K. Ono,
\newblock{\it Unearthing the visions of a master: harmonic Maass forms and number theory,}
\newblock{Proceedings of the 2008 Harvard-MIT Current Developments in Mathematics Conference, International Press, Somerville, MA, 2009, 347--454.}

\bibitem{Pri_SmlPosWgt_II}
W. Pribitkin,
\newblock {\it The {F}ourier coefficients of modular forms and {N}iebur modular
  integrals having small positive weight {II},}
\newblock Acta Arith {\bf 93} (2000), no. 4, 343--358.

\bibitem{Rad_FuncEqnModInv}
H. Rademacher,
\newblock {\it The {F}ourier {S}eries and the {F}unctional {E}quation of the
  {A}bsolute {M}odular {I}nvariant {J}({$\tau$}),}
\newblock Amer. J. Math. {\bf 61} (1939), no. 1, 237--248.

\bibitem{MR0364103}
H. Rademacher,
\newblock {\em Topics in analytic number theory,}
\newblock Springer-Verlag, New York-Heidelberg, 1973.
\newblock Edited by E. Grosswald, J. Lehner and M. Newman, Die Grundlehren der
  mathematischen Wissenschaften, Band 169.
  
\bibitem{RademacherLogEta}
H. Rademacher,
\newblock{ \em On the transformation of $\log \eta(\tau)$,}
\newblock{ J. Indian Math Soc. {\bf 19} (1955), 25--30.}

\bibitem{Roe_EigPrbAutFrmsHypPln}
W.~Roelcke, ``Das {E}igenwertproblem der automorphen {F}ormen in der
  hyperbolischen {E}bene. {I}, {II},'' {\em Math. Ann. 167 (1966), 292--337;
  ibid.} {\bf 168} (1966)  261--324.


\bibitem{Sar_ANT}
P.~Sarnak, \href{http://dx.doi.org/10.1007/BFb0071548}{``Additive number theory
  and {M}aass forms,''} in {\em Number theory ({N}ew {Y}ork, 1982)}, vol.~1052
  of {\em Lecture Notes in Math.}, pp.~286--309.
\newblock Springer, Berlin, 1984.
\newblock \url{http://dx.doi.org/10.1007/BFb0071548}.


\bibitem{Sel_EstFouCoeffs}
A.~Selberg, ``On the estimation of {F}ourier coefficients of modular forms,''
  in {\em Proc. {S}ympos. {P}ure {M}ath., {V}ol. {VIII}}, pp.~1--15.
\newblock Amer. Math. Soc., Providence, R.I., 1965.


\bibitem{MR1074485}
N. Skoruppa,
\newblock {\it Explicit formulas for the {F}ourier coefficients of {J}acobi and elliptic modular forms,}
\newblock Invent. Math. {\bf 102} (1990), no. 3, 501--520.

\bibitem{MR822245}
S. Smith,
\newblock {\it On the head characters of the {M}onster simple group,}
\newblock in {\em Finite groups---coming of age ({M}ontreal, {Q}ue., 1982)},
  volume~45 of {\em Contemp. Math.}, pp. 303--313. Amer. Math. Soc.,
  Providence, RI, 1985.

\bibitem{Tho_NmrlgyMonsEllModFn}
J. Thompson.
\newblock {\it Some numerology between the {F}ischer-{G}riess {M}onster and the
  elliptic modular function,}
\newblock Bull. London Math. Soc. {\bf 11} (1979), no. 3, 352--353.

\bibitem{Witten2007}
E. Witten,
\newblock {\it Three-Dimensional Gravity Revisited,}
\newblock preprint.

\bibitem{zagier_mock}
D. Zagier,
\newblock {\it Ramanujan's mock theta functions and their applications (d'apr\`es
  {Z}wegers and {O}no-{B}ringmann),}
\newblock S\'eminaire Bourbaki, 60\`eme ann\'ee, 2007-2008, no. 986, Ast\'erisque {\bf 326} (2009), Soc. Math. de France, 143--164.

\bibitem{zwegers}
S. Zwegers,
\newblock {\em {Mock Theta Functions},}
\newblock PhD thesis, Utrecht University, 2002.

\end{thebibliography}

\end{document}